\documentclass[12pt]{amsart}
\usepackage{amssymb}\usepackage{color}

\usepackage{marginnote}

\usepackage{url,titletoc,enumerate}

\usepackage[pagebackref,colorlinks,linkcolor=red,citecolor=blue,urlcolor=blue,hypertexnames=true]{hyperref}
  \def\ol{\overline}

    \usepackage{graphicx,tikz}
    \usetikzlibrary{shapes,patterns,calc}

  \usepackage{mathrsfs,accents,setspace}
  
\numberwithin{equation}{section}

\newtheorem{thm}{Theorem}[section] 
\newtheorem{theorem}[thm]{Theorem}

\newtheorem{cor}[thm]{Corollary}
\newtheorem{lemma}[thm]{Lemma}
\newtheorem{lem}[thm]{Lemma}
\newtheorem{prop}[thm]{Proposition}

\theoremstyle{definition}
\newtheorem{definition}[thm]{Definition}
\newtheorem{defn}[thm]{Definition}
\newtheorem{conj}         [thm]{Conjecture}

\newtheorem{remark}[thm]{Remark}
\newtheorem{rem}[thm]{Remark}
\newtheorem{example}[thm]{Example}

  \newcommand{\cD}{\mathcal D}
   \newcommand{\Abullet}{\hspace{10pt} \raisebox{2pt}{\scriptsize{\textbullet}}}

\def\tX{\tilde X}

\def\mcop{\textrm{\rm co$^{\rm oper}\cD_p$}}
\def\mcoop{\textrm{\rm co$\,\cD^\infty_p$}}
\def\mco{\textrm{\rm co$^{\rm mat}$}}
\def\cmco{\overline{\textrm{\rm co}}^{\rm mat}}

\def\bTV{bent TV screen}

\def\tvs{\textvisiblespace}

\def\ben{\begin{enumerate}}
\def\een{\end{enumerate}}

\newcommand{\sac}[1]{ \mbS_{#1} }
\newcommand{\sar}[1]{ \mbS_{#1} }

\def\bem{\begin{pmatrix}}
\def\eem{\end{pmatrix}}

\def\beq{\begin{equation}}
\def\eeq{\end{equation}}

\newcommand{\hs}{H^{\Uparrow}}

  \newcommand{\inter}[1]{\accentset{\smash{\raisebox{-0.12ex}{$\normalsize\circ$}}}{#1}\rule{0pt}{2.3ex}}

\renewcommand{\subset}{\subseteq}
\renewcommand{\supset}{\supseteq}
\renewcommand{\emptyset}{\varnothing}

\newcommand{\ax}{\langle x\rangle}

\DeclareMathOperator{\conv}{co}
\DeclareMathOperator{\proj}{proj}
\DeclareMathOperator{\rank}{rank}

 \def\tX{\tilde X}

\newcommand{\df}[1]{{\bf{#1}}{\index{#1}}}

\def\cC{\mathcal C}

\def\cK{\mathscr K}

\def\ccD{\inter{\cD}}

\def\cP{\mathscr P}

\def\RR{\mathbb R}
\def\al{\alpha}
\def\be{\beta}
\def\ga{\gamma}
\def\de{\delta}
\def\epsilon{\varepsilon}
\def\eps{\epsilon}

\def\bec{\begin{conj}}
\def\eec{\end{conj}}
\def\bex{\begin{example}}
\def\eex{\end{example}}
\def\de{\delta}

\def\Lift{\mathcal L}

\def\Liftfin{\mathcal L^{\rm fin}}

\newcommand{\hLift}{\hat{\Lift}}
\def\hLiftfin{\hat{\mathcal L}^{\rm fin}}

\def\mbS{{\mathbb S}}
\def\mS{{\mathbb S}}
\def\smat{\mS}
\def\smatmg{\mathbb S_m^g}

\def\smatg{\mathbb S^g}
\def\smatn{\mathbb S_n}
\def\SRnn{\mathbb S_n}

\def\smatng{\mathbb S_n^g}
\def\smatnh{\mathbb S_n^h}

\def\smatg{\mathbb S^g}
\def\smath{\mathbb S^h}

\def\hY{\hat{Y}}
\def\cS{\mathcal S}
\def\cH{\mathscr H}
\def\cE{\mathcal E}
\def\cN{\mathcal N}
\def\cO{\cD^\infty}
\def\smatog{\mathbb S_{\rm oper}^g}
\def\smatkg{\mathbb S_K^g}

\def\LK{{\rm Lat}(\cK)}

\def\RR{ {\mathbb{R}} }
\def\R{ {\mathbb{R}} }
\def\C{ {\mathbb{C}} }

\def\N{ {\mathbb{N}} }

\def\epsilon{\varepsilon}
\def\eps{\epsilon}

\def\La{\Lambda}

\def\ll{\ell}
\def\cL{\mathcal L}
\def\dpovertwo{\left\lceil \frac12\deg(p)\right\rceil}

\def\FF{\mathbb C}

\def\FFnn{\FF^{n\times n}}

\def\mt{\nu}
\def\x{x}
\def\gP{\mathfrak P}

\begin{document}

\title[Matrix Convex Hulls of Free Semialgebraic Sets]{Matrix Convex Hulls of \\[.1cm] Free Semialgebraic Sets}

\author[J.W. Helton]{J. William Helton${}^1$}
\address{J. William Helton, Department of Mathematics\\
  University of California \\
  San Diego}
\email{helton@math.ucsd.edu}
\thanks{${}^1$Research supported by the National Science Foundation (NSF) grant
DMS 1201498, and the Ford Motor Co.}

\author[I. Klep]{Igor Klep${}^{2}$}
\address{Igor Klep, Department of Mathematics, 
The University of Auckland, New Zealand}
\email{igor.klep@auckland.ac.nz}
\thanks{${}^2$Supported by the Faculty Research Development Fund (FRDF) of the
University of Auckland (project no. 3701119). Partially supported by the Slovenian Research Agency grant P1-0222.}

\author[S. McCullough]{Scott McCullough${}^3$}
\address{Scott McCullough, Department of Mathematics\\
  University of Florida\\ Gainesville 
   }
   \email{sam@math.ufl.edu}
\thanks{${}^3$Research supported by the NSF grant DMS 1101137}

\date{\today}

\setcounter{tocdepth}{3}
\contentsmargin{2.55em} 
\dottedcontents{section}[3.8em]{}{2.3em}{.4pc} 
\dottedcontents{subsection}[6.1em]{}{3.2em}{.4pc}
\dottedcontents{subsubsection}[8.4em]{}{4.1em}{.4pc}

\makeatletter
\newcommand{\mycontentsbox}{%
{\centerline{NOT FOR PUBLICATION}
\addtolength{\parskip}{-2.3pt}
\small\tableofcontents}}
\def\enddoc@text{\ifx\@empty\@translators \else\@settranslators\fi
\ifx\@empty\addresses \else\@setaddresses\fi
\newpage\mycontentsbox\newpage\printindex}
\makeatother


\begin{abstract}\scriptsize
This article resides in the realm of the noncommutative (free) analog of
real algebraic geometry -- the study  
  of polynomial inequalities and equations over the real numbers -- 
  with a focus on matrix convex sets $\cC$ and their projections $\hat \cC$.   
  A free  semialgebraic set which is convex as well as  
  bounded and open can be represented as the solution set
  of a Linear Matrix Inequality (LMI),  a result which suggests that 
  convex free  semialgebraic sets are rare. Further, Tarski's transfer principle
  fails in the free setting:   The projection of a free convex semialgebraic set need not be free semialgebraic.
  Both of these results, and the importance of convex approximations in the optimization 
 community,  provide impetus and motivation  for 
  the study of the matrix convex hull of free  semialgebraic sets.

This article presents the construction of a sequence ${\cC}^{(d)}$ of LMI domains
in increasingly many variables  whose projections $\hat{\cC}^{(d)}$ are successively finer 
outer approximations of the
matrix convex hull of a free semialgebraic set $\cD_p=\{X: p(X)\succeq0\}$.
It is based on free analogs of moments and Hankel
matrices. Such an approximation scheme is possibly the best that can be done
in general. 
Indeed, natural noncommutative transcriptions  
 of formulas for certain well-known classical (commutative) convex hulls 
 do not produce the convex hulls in the free case. 
 This failure is illustrated here on one of the simplest free nonconvex $\cD_p$.

A basic question is which free sets $\hat \cS$ are the projection of a
free semialgebraic set $\cS$?  Techniques and results of this paper
bear upon this  question which is open even for convex sets.
\end{abstract}

\maketitle

\newpage

\section{Introduction}
 This  article resides  within the realm  of the recently emerging area of
noncommutative (free) real algebraic geometry. As such
 it concerns free noncommutative polynomials $p(x)=p(x_1,\ldots,x_g)$, and 
their associated free semialgebraic sets $\cD_p$ (resp. $\mathfrak P_p$) consisting of those
  $g$-tuples of self-adjoint  matrices $X$ of the same size for which 
 $p(X)$ is positive semidefinite (resp. definite). 
 The case of (matrix) convex  $\cD_p$ 
 is important in  applications and also 
 serves as an entr\'ee to basic general aspects of free real algebraic
geometry. 

From the  main result of \cite{HM12}, 
a  bounded and open 
free   semialgebraic set that is  convex can be represented 
 as the set of solutions  to a Linear Matrix Inequality (LMI),
 called a free spectrahedron.
This result is decidedly negative from the viewpoint of systems engineering,
since it means that convex free semialgebraic sets are rare. 
It also motivates the theme of this article,  the challenging problem of understanding 
the  convex hull of a free   semialgebraic set $\cD_p.$\looseness=-1

While formal definitions occur later, we now give  the basic flavor of our results.
The main classical approach for producing
the convex hull of a basic semialgebraic set $\cD \subset \RR^g$ 
is to cleverly construct  a   spectrahedron $\cC$ in a bigger space
whose projection 
 onto $\RR^g$ is the convex hull of $\cD.$ In the literature the set 
$\cC$ goes by several names. Here we will refer to these as an \df{LMI lift} or \df{spectrahedral lift} of the convex
hull of $\cD$.  Developing the free analog of a theorem due to Lasserre 
for classical semialgebraic sets \cite{Las}, 
under modest hypotheses on $\cD_p,$
{\it we construct  
a sequence $\cC^{(d)}$ of free spectrahedra 
in larger and larger spaces whose projections close down on 
the free convex hull of $\cD_p$.} See Corollary \ref{cor:seqlasserre}.

 We remark that solutions sets of LMIs
 play a prominent role in the theory of completely positive maps 
 and operator systems \cite{Arv72,Pau02} as well as quantum information theory (see for instance \cite{JKPP11}).
  Moreover, their projections are related to recent advances in the theory of quotients of operator
 systems for which \cite{FP} is one of several recent references. 
 A natural approach to understanding convexity in the free setting  is through the study
 of free analogs of extreme points. One such is Arveson's \cite{Arv72} notion of a {\it boundary representation}
 as a noncommutative analog of a peak point for a uniform algebra. 
 As an emphatic culmination of a spate of recent activity,
 the article \cite{DK+} validates
 Arveson's vision that an  operator system has  sufficiently many boundary representations 
 to generate its $C^*$-envelope.  For matrix convex hulls of free semialgebraic sets 
 other notions of extreme points occur naturally (see for instance \cite{F04,WW99,Kls13}) and 
 are treated in the forthcoming article \cite{Arvboundary}.\looseness=-1

Beyond this point the news is bad.
An approximation scheme, like that found here,  is possibly the best that can be done in general.
As evidence, 
we study thoroughly a $\cD_p$ which has a strong  claim to the title of simplest  nonconvex free semialgebraic set.
The free analogs of 
two different classical spectrahedral  lifts  $\cC$ for $\cD_p$ each have the property that 
the projection $\hat\cC$ of $\cC$ is convex and 
contains $\cD_p$ and, at the scalar (commutative) level $\hat\cC(1)=\cD_p(1).$ However, in both cases, 
$\hat\cC$ is not the free
convex hull of $\cD_p$; that is, $\mco\cD_p\subsetneq \hat\cC$.
 See Example \ref{ex:btvlift},
 Subsection \ref{subsec:hankelFun} and Section \ref{sec:ex}.

A cornerstone of classical real algebraic geometry (RAG) is 
Tarski's transfer principle: 
the projection of a 
semialgebraic set
is again semialgebraic. 
In free RAG the corresponding assertion is false even for convex sets, see \cite{HM12}.
 Thus 
a basic question, on which this article bears and which is perhaps the most accessible path to 
understanding the class of sets closed with respect to projections, 
and containing the free semialgebraic sets,
is which free  sets are the projection of a free spectrahedron.

\subsection{Context and Perspective}
The standard reference on classical RAG is \cite{BCR98}. Two 
more tailored to our purposes are \cite{Las09b} and  \cite{Lau09}.

The construction  of lifts used here is analogous to
one introduced by Lasserre \cite{Las} and Parrilo \cite{Par06} independently.
It involves positivity for  multivariable moment matrices, studied systematically
 by Curto and Fialkow in a series of articles (see for example \cite{CF}),  as well as their duals which
are algebraic certificates of  positivity for  polynomials,
called Positivstellens\"atze.
Lasserre's key idea was to use a Positivstellensatz representation of linear functionals
$\ell$
delineating the convex hull of the set $\cD$ under study.
When a nice Positivstellensatz exists for \emph{all} such $\ell$, one gets that a 
suitable spectrahedron $\cC$, whose projection equals $\cD$, exists.
In fact, a related idea is that of the theta body introduced earlier to combinatorial
optimization by Lov\'asz in \cite{Lov79}; see also \cite{GLS93}.
The recent survey \cite{GT12} of  Gouveia and Thomas ties these subjects together. See also 
their papers with Laurent and Parrilo \cite{GLPT12,GPT10,GPT12}.
LMI lifts of convex sets appeared in the book of 
Nesterov and Nemirovskii 
 \cite{NN94} at the outset of SDP. In their examples of sets with LMI representations -- 
see Chapter 6 -- rather than representing the sets, they gave representations for the lifts.

Returning to free lifts we mention that they are used in linear systems 
engineering to obtain free convex envelopes of sets.
In the absence of any systematic theory, the literature  consists of 
 clever constructions (cf. \cite{OGB02,GO10}).
Moment matrix positivity in a free noncommutative context was 
studied in \cite{PNA10},
in connection with noncommutative sums of squares,  following \cite{HM04a}
and focusing on  computational aspects; see also \cite{HKM12}.

 While the setup of this paper is complex, that is, we work with self-adjoint
 complex matrices, the results carry over with little change to a combination of real symmetric and skew-symmetric 
  matrices, cf.~Remark \ref{rem:CvsR}.

 We thank Cory Stone and \v Spela \v Spenko for many helpful comments on early versions of this manuscript.

\subsection{Guide to the Paper}
\begin{enumerate}[$\Abullet$]
 \item
Section \ref{sec:2} contains basic definitions, including that of 
free   polynomials, free semialgebraic sets, free convexity, and 
the matrix and operator convex hull of a free semialgebraic set. 
 \item
Section \ref{sec:cp} concerns  
 linear pencils and their relation to matrix convex hulls.
  \item
 Basic properties of  {projections of free spectrahedra}
   are presented in Section \ref{sec:4}.
  \item
   For a given free semialgebraic set $\cD_p$,  the construction of Section  \ref{sec:freelassy}, based upon 
  free analogs of moment sequences and Hankel matrices,   produces an 
 infinite free spectrahedron $\cC$ together with a projection from $\cC$ onto 
the operator convex hull $\mcop$ of
 $\cD_p.$\looseness=-1
  \item
  In Section \ref{sec:preEx},  truncations of the free Hankel matrices from Section \ref{sec:freelassy} 
  which in turn produce a sequence $(\cC^{(d)})_d$ of (finite) free spectrahedra together with projections
 $\pi_d$ are introduced. It is shown that  $\pi_d(\cC^{(d)})$ produces successively better outer approximations 
  to $\mcop$ and, in the limit, converges to $\mcop$. 
 \item
  Examples appear in Section \ref{sec:ex}. 
 \end{enumerate}
 
%
%

\section{Free Sets and Free Polynomials} \label{sec:2}
 Fix a positive integer $g$. 
 For a positive integer $n$, let $\smatng$ denote the set of $g$-tuples of
 complex self-adjoint
 $n\times n$ matrices and let $\smatg$ denote
 the  sequence $(\smatng)_n$. 
 A  \df{subset} $\Gamma$ of $\smatg$ is a sequence
 $\Gamma=(\Gamma(n))_n$ where
  $\Gamma(n) \subset \smatng$ for each $n$.  The subset $\Gamma$
  is \df{closed with respect to direct sums}  if 
 $A=(A_1,\dots,A_g)\in\Gamma(n)$ and
  $B=(B_1,\dots,B_g)\in\Gamma(m)$ implies
  \begin{equation}
  \label{eq:dsumsed}
  A\oplus B := \left ( \begin{pmatrix} A_1 & 0 \\ 0 & B_1\end{pmatrix}, \dots, \begin{pmatrix} A_g & 0\\ 0 & B_g\end{pmatrix} \right )\in\Gamma(n+m).
\end{equation}
  It is closed with respect to 
  \df{(simultaneous) unitary  conjugation} if for each $n,$ each $A\in \Gamma(n)$
 and each $n\times n$ unitary matrix $U$,
\[
 U^* A U = (U^* A_1U,\dots, U^* A_gU)\in \Gamma(n). 
\]
 The set $\Gamma$ is a  \df{free set} if it is closed with respect to direct sums and 
simultaneous  unitary conjugation.
We refer the reader to 
 \cite{Voi04,Voi10,KVV+,MS11,Pope10,AM+,BB07}
for a systematic study of free sets and free function theory.

We call a  free  set $\Gamma$
 \df{(uniformly) bounded} if there is a $C\in\R_{>0}$ such that $C-\sum X_j^2 \succeq 0$
  for all $X\in\Gamma$.  

\subsection{Free Polynomials}

\subsubsection{Words and free polynomials}
  We write $\ax$ for the monoid freely
generated by $x=(x_1,\ldots, x_g)$, i.e., $\ax$ consists of \df{words} in the $g$
noncommuting
letters $x_{1},\ldots,x_{g}$
(including the {\bf empty word $\emptyset$} which plays the role of the identity).
Let $\FF\ax$ denote the associative
$\FF$-algebra freely generated by $x$, i.e., the elements of $\FF\ax$
are polynomials in the freely noncommuting variables $x$ with coefficients
in $\FF$. Its elements are called \df{free polynomials}.
Endow $\FF\ax$ with the natural \df{involution} ${}^*$
which extends the complex conjugation on $\C$, fixes $x$, reverses the
 order of words, and acts $\R$-linearly on polynomials. 
Polynomials fixed under this involution are \df{symmetric}.
The
length of the longest word in a free polynomial $f\in\FF\ax$ is the
\df{degree} of $f$ and is denoted by $\deg( f)$ or $|f|$ if $f\in\ax$. The set
of all words of degree at most $k$ is $\ax_k$, and $\FF\ax_k$ is the vector
space of all free polynomials of degree at most $k$.

 Fix positive integers
 $\mt$ and $\ell$.  \df{Free matrix polynomials} -- 
  elements of 
$\FF^{\ell\times\mt}\ax=\FF^{\ell\times\mt}\otimes \FF\ax;$ 
 i.e.,
 $\ell\times\mt$ matrices with entries from $\FF\ax$ -- will play a role
 in what follows. Elements of $\FF^{\ell\times\mt}\ax$ are 
 represented as
\beq\label{eq:preeq0}
  P=\sum_{w\in\ax} B_w w\in\FF^{\ell\times\mt}\ax,
\eeq
  where $B_w\in \FF^{\ell\times\mt}$, and the
 sum is finite. 
The involution ${}^*$ extends to matrix polynomials by
\[
  P^* =\sum_w B_w^*  w^*\in\FF^{\mt\times\ell}\ax. 
\]
  If $\mt=\ell$ and $P^*=P$, we say $P$ is \df{symmetric}.

\subsubsection{Polynomial evaluations}
 If $p\in\FF\ax$ is a free polynomial and  $X\in\SRnn^g$, then 
 the evaluation $p(X)\in\FFnn$ is defined in the natural way by 
 replacing $x_{i}$ by $X_{i}$ and sending the empty word to
  the appropriately sized identity matrix.
Such evaluations produce finite dimensional $*$-representations of
the algebra of free polynomials and vice versa.

Polynomial evaluations extend to matrix polynomials by evaluating 
entrywise. That is, if $P$ is as in \eqref{eq:preeq0}, then
\[
P(X) = \sum_{w\in\ax} B_w \otimes w(X) \in \FF^{\ell n\times \mt n},
\]
where $\otimes$ denotes the (Kronecker) tensor product.
 Note that if $P\in\FF^{\ell\times \ell}\ax$ is symmetric, and $X\in      \SRnn^{g}$, then
 $P(X)\in\FF^{\ell n\times \ell n}$ is a self-adjoint matrix.

\subsection{Free Semialgebraic Sets}
A symmetric free polynomial and even a symmetric matrix  polynomial $p$ 
 in free  variables naturally determine free sets  \cite{dOHMP09}
  via
 \[
 \begin{split}
    \cD_p(n):= & \{ X \in \smatng : p(X) \succeq 0\}, \qquad \quad \cD_p: = (\cD_p(n))_n .
 \end{split}
\] 
  By analogy with real algebraic geometry \cite{BCR98}, 
  we will refer to these as \df{free (basic closed) semialgebraic sets}.

\begin{example}\rm
\label{ex:btv}
Consider 
\begin{equation}\label{eq:holyTV}
p=1-x_1^2-x_2^4.
\end{equation}
In this case $p$ is 
 symmetric with $p(0)=1>0$.  
The free semialgebraic set $\cD_p$ is called the
 \df{bent free TV screen}, or (bent) TV screen for short.
We shall use this example at several places to illustrate the developments
in this paper.

\def\nos{100}
\begin{center}
\begin{tikzpicture}[domain=-2:2,scale=2] 
\draw[very thin,color=gray] (-1.6,-1.6) grid (1.6,1.6);
\draw[color=black, domain=1:-1, very thick,  samples=\nos] plot (\x,sqrt{(sqrt{(1-(\x)^2)})});
\draw[color=black, domain=1:-1, very thick,  samples=\nos] plot (\x,-sqrt{(sqrt{(1-(\x)^2)})});
\draw[color=black, domain=-1:1, very thick,  samples=\nos] plot (\x,sqrt{(sqrt{(1-(\x)^2)})});
\draw[color=black, domain=-1:1, very thick,  samples=\nos] plot (\x,-sqrt{(sqrt{(1-(\x)^2)})});
 \fill[color=blue!20, domain=1:-1, samples=\nos] plot (\x,sqrt{(sqrt{(1-(\x)^2)})}); 
 \fill[color=blue!20, domain=1:-1, samples=\nos] plot (\x,-sqrt{(sqrt{(1-(\x)^2)})});
 \fill[color=blue!20, domain=-1:1, samples=\nos] plot (\x,sqrt{(sqrt{(1-(\x)^2)})}); 
 \fill[color=blue!20, domain=-1:1, samples=\nos] plot (\x,-sqrt{(sqrt{(1-(\x)^2)})});
\fill[blue!20] (0,0) circle (1);
\draw[->] (-1.7,0) -- (1.7,0) node[right] {$x_1$}; \draw[->] (0,-1.7) -- (0,1.7) node[above] {$x_2$};
\filldraw[color=black] (1,0) circle (0.02) node[below right] {$1$};
\filldraw[color=black] (0,1) circle (0.02) node[above left] {$1$};
\end{tikzpicture} 
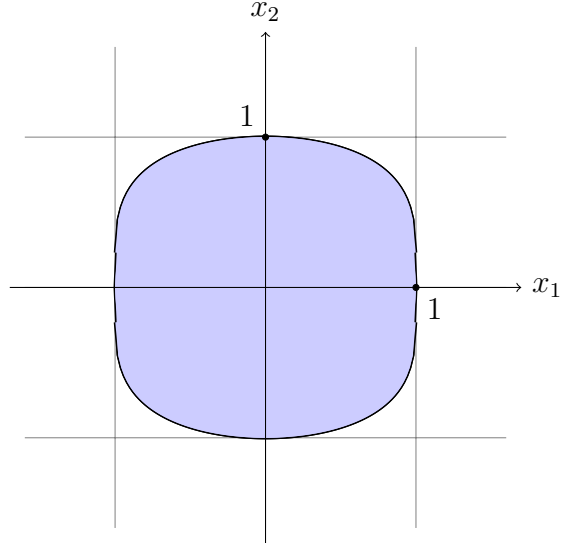
\end{center}~\centerline{Bent TV screen $\cD_p(1)=\{ (x_1,x_2)\in\R^2 : 1-x_1^2-x_2^4\geq0\}$.}
\end{example}

  A subset $\Gamma$ of $\smatg$ is \df{closed with respect to restriction to reducing subspaces}
  if $A\in\smatng$ and $H\subset \FF^n$ is an invariant (reducing) subspace for $A$ implies 
  that $A$ restricted to $H$ is in $\Gamma.$

\begin{lemma}
 \label{lem:semialgsarefree}\rm
 \mbox{}
\par
\ben[\rm(1)]
\item
  For each $n,$ the 
  set $\cD_p(n)$ is a semialgebraic subset of $\smatng$.
\item
  The free semialgebraic set $\cD_p$ is a free set. Moreover, it
  is   closed with respect to restriction to reducing subspaces.
   \een
\end{lemma}
  
\begin{proof}
   Fix $n$.  There are scalar commutative polynomials $p_{i,j}$ 
  in $gn^2$ variables such that  $p(X) = (p_{i,j}(X))$ for $X\in\smatng$.
    By Sylvester's criterion, $p(X)\succeq 0$ if and only if
   all the principal minors of $p(X)$ are nonnegative.
   Since these minors are all polynomials,
 it follows that $\cD_p(n)$ is a semialgebraic set.

   It is evident that  $\cD_p$ is  a free set.
   Suppose $H$ reduces $A\in \cD_p(n)$. In this case,
   $A=A^1\oplus A^2$ for $A^j\in \mathbb S_{n_j}^g$
  with $n_1+n_2=n$.  Since $0\preceq p(A)=p(A^1)\oplus p(A^2),$ it follows that 
  $p(A^j)\succeq 0$ for each $j$.  Hence $A\in\cD_p(n_1)$ and 
 $\cD_p$ is closed with respect to restrictions to reducing subspaces.
\end{proof}

 \subsection{Free Convexity}
 
 A set $\Gamma = (\Gamma(n))_n \subset \smatg$ 
is \df{matrix convex} or \df{freely convex} if it is closed
under direct sums and \df{(simultaneous) isometric conjugation}; i.e.,
 if for each $m\le n$, each   $A=(A_1,\dots,A_g)\in\Gamma(n)$, and
 each isometry $V:\FF^m\to\FF^n$,
\[
 V^* A V := \left ( V^*A_1V, \dots, V^* A_g V\right )\in\Gamma(m)
\]
 In particular, a matrix convex set is  a free set. 

In the case that $\Gamma$ is matrix convex, 
it is easy to show that each $\Gamma(n)$ is itself convex. 
Indeed, given real numbers $s,t$ with $s^2+t^2=1$  and $ X, Y \in \Gamma(n)$, let
\[
  V=\begin{pmatrix} 
 s I_n \\ t I_n
 \end{pmatrix}
\]
 and observe that
\begin{equation}\label{eq:mconv}
  V^* \begin{pmatrix} X & 0 \\ 0 & Y \end{pmatrix} V = s^2 X + t^2 Y \in \Gamma(n).
\end{equation}
 More generally, if $A^\ell=(A^\ell_1,\dots,A^\ell_g)$ are in $\Gamma(n_\ell)$, then 
  $A=\bigoplus_\ell A^\ell\in \Gamma(n),$
  where $n=\sum n_\ell$.
 Hence, if 
\[
   V= \begin{pmatrix} V_1 \\ V_2\\ \vdots \\ V_k\end{pmatrix}
\]
 is an isometry and $V_\ell$ are $n_\ell \times m$ matrices (for some $m$), then
\begin{equation}
\label{eq:ccomb}
 V^* A V = \sum^k_{\ell=1} V_\ell^* A^\ell V_\ell \in \Gamma (m) \qquad \text{where}
 \quad  \sum^k_{\ell=1} V_\ell^* V_\ell = I  
\end{equation}
 A sum as in 
  \eqref{eq:ccomb} is a 
 \df{matrix convex combination} of the 
 $g$-tuples $\{A^\ell: \ \ell =1, \dots, k \}$.
    
\begin{lemma}
 \label{lem:capped}\rm
  Suppose  $\Gamma$ is a free subset of $\smatg$. 
\ben[\rm(1)]
 \item
  If  $\Gamma$ is closed with respect to  restriction to reducing subspaces,
  then the following are equivalent: 
 \ben[\rm (i)]
  \item  $\Gamma$ is matrix convex; 
  \item  each  $\Gamma(n)$ is convex in the  classical  
  sense of taking scalar convex combinations.

 \een
\item
  If $\Gamma$ is $($nonempty and$)$ matrix convex, then $0\in \Gamma(1)$
  if and only if  $\Gamma$ is closed with respect to $($simultaneous$)$ conjugation by
  contractions.
\een
\end{lemma}

\begin{proof}{}
  Evidently (i) implies (ii).  The implication (ii) implies (i) is proved 
  in \cite[\S2]{HM04a}.
 For item (2), if $\Gamma$ is closed with respect to conjugation by a contraction,
  then given an $A\in\Gamma(n)$, letting $z:\FF\to\FF^n$ 
be the zero mapping,
  gives
  $z^* A z = 0\in \FF^g$. Hence, $0\in\Gamma(1)$.   Conversely,
  suppose $0\in\Gamma(1)$. In this case for each $n$ the zero tuple $0_n$ is in
  $\Gamma(n)$ as $\Gamma$ is closed with respect to direct sums. 
 Given 
 an $n\times n$ contraction $F$, and
 $X\in\Gamma(n)$ observe that $X\oplus 0_n\in\Gamma(2n),$  form the isometry 
\[
  V^*=\begin{pmatrix} F^* & (I-F^*F)^{\frac12} \end{pmatrix}
\]
  and compute 
\[
   F^* X F = V^* \begin{pmatrix} X & 0 \\ 0 & 0 \end{pmatrix} V \in\Gamma. \qedhere
\]
\end{proof}

\begin{remark}\rm
 Combining the second items of Lemmas \ref{lem:semialgsarefree} and
  \ref{lem:capped} it follows that the free semialgebraic set
  $\cD_p$ is matrix convex if and only if
  each $\cD_p(n)$ is convex. 
\end{remark}

\begin{example}
\label{ex:btvv}
Consider the TV screen given by $p=1-x_1^2-x_2^4$ introduced in Example \ref{ex:btv}.
While $\cD_p(1)$ is convex (see Example \ref{ex:btv} for a picture), 
it is known that $\cD_p$ is not 
matrix convex, see 
\cite{DHM07}
or \cite[Chapter 8]{BPR13}. Indeed, already $\cD_p(2)$ 
is not a convex set.
\end{example}

\smallskip

\subsection{The Matrix Convex Hull}
The \df{matrix convex hull} of a subset $\Gamma = (\Gamma(n))_n$ of $\smatg,$ 
 denoted $\mco \Gamma$, 
is the smallest matrix convex set containing $\Gamma$. 
As usual, the intersection of matrix convex sets is matrix convex, 
so the notion of a hull is well defined. 
 Further, 
there is a simple description of the matrix convex hull of a free  set.

 For positive integers $n$ let 
\begin{equation}\label{eq:Cp}
 \cC(n) :=\bigcup_{m\in\N}\{X\in\smatng: X=V^* Z V \mbox{ for some isometry $V\in\FF^{m\times n}$ and } Z\in \Gamma(m) \}.
\end{equation}
 In the case that $\Gamma$ is closed with respect to direct sums, 
 it is straightforward to verify that $\cC= \big(\cC(n)\big)_n$ is a matrix convex set which contains
  $\Gamma$.  On the other hand, $\cC$ must be contained in any matrix
 convex set containing $\Gamma$.
Hence we conclude:

\begin{prop}
 \label{prop:semialghull}\rm
 If $\Gamma$ is closed with respect to direct sums, then 
$\cC$ is 
its matrix convex hull.\looseness=-1
\end{prop}

\subsection{Topological Properties of the Matrix Convex Hull}
  A natural norm on $\smatng$ is given by
 \[
  \|X\|^2 =\sum_{j=1}^g \|X_j\|^2
 \]
 for  $X=(X_1,\dots,X_g)\in \smatng.$

\label{sec:topo}
  The subset $\cS=(\cS(n))_n$ of $\smatg$ is \df{open} if each $\cS(n)$ is open. 

\begin{lemma}\label{lem:open}\rm
 If the open set $\cS\subseteq\smatg$ is closed with respect to direct sums, 
 then $\mco \cS$ is open.
\end{lemma}

\begin{proof}
  To show that $\mco \cS(n)$ is open, let $X\in \mco \cS(n)$ be given. By Proposition \ref{prop:semialghull},
 there exists an $m\in\N,$ a $Z\in\cS$ and an isometry $V:\FF^n\to\FF^m$ such that
 $X=V^*ZV$.  Because $\cS(m)$ is open, there exists an $\epsilon>0$ such that
 if $\|W-Z\|<\epsilon$, then $W\in \cS(m)$.   Now suppose $Y\in\smatng$ and $\|Y-X\|<\epsilon$. 
 Writing,
\[
  Z = \begin{pmatrix} X & \beta \\ \beta^* & \delta \end{pmatrix}
\]
  with respect to the decomposition of $\FF^m$ as the range of $V$ direct sum its orthogonal 
 complement, let
\[
  W = \begin{pmatrix} Y & \beta \\ \beta^* & \delta \end{pmatrix}.
\]
 Thus $W\in\cS(m)$ and, by another application of Proposition \ref{prop:semialghull}, 
  $Y= V^* W V \in \mco \cS(n)$. Hence $\mco \cS(n)$ is open. 
\end{proof}

  Let $\cmco \Gamma$ denote the closure of the convex hull of the free set $\Gamma$,
    i.e.,
  \[
  \cmco \Gamma = \left ( \overline { \mco \Gamma(n)} \right )_n.
  \]

\begin{lem}
 \label{lem:closureofmco}\rm
 If $K=(K(n))_n$ is a matrix convex set, then $\overline{K}=(\overline{K(n)})_n$ is
 also matrix convex. Here $\overline{K(n)}$ is the closure of $K(n)$
 in $\smatng$.
\end{lem}

\begin{proof}
  To see that $\overline{K}$ is closed with respect to direct sums, 
  suppose $X\in\overline{K(n)}$ and $Y\in\overline{K(m)}$. 
  There exists sequences $(X^\ell)$ and $(Y^\ell)$ from
 $K(n)$ and $K(m)$ converging to $X$ and
 $Y$ respectively. It follows that $X^\ell\oplus Y^\ell \in K(n+m)$
 converges to $X\oplus Y$ and thus $X\oplus Y\in \overline{K(n+m)}$.

  To see that $\overline{K}$ is closed with respect to simultaneous
 isometric conjugation, suppose $X\in \overline{K(n)}$ and
  $V:\FF^m\to\FF^n$ is an isometry.  There exists
  a sequence $(X^\ell)$ from $K(n)$ which converges to $X$.
 Thus, the sequence $(V^* X^\ell V)$ lies in $K(m)$ and converges to
  $V^*XV$. Thus $V^*XV \in \overline{K(m)}$ and the proof is complete.
\end{proof}

\begin{lem}
\label{lem:closed}\rm
  Suppose $\Gamma$ is a free set. If each $\Gamma(n)$ is compact, then 
  for each $m$, $\mco\Gamma(m)$ is naturally a nested increasing union
 of compact convex sets. 
\end{lem}

\begin{proof}
For each $n\ge m$, let
   $$P_n(m)=\{V^* X V\mid V:\FF^m\to \FF^n \text{ is an isometry}, \   X\in\Gamma(n)\}\subseteq\mco\Gamma(m).$$
  Let $E_n(m)\subseteq \mco \Gamma(m)$ denote the (ordinary) convex hull of $P_n(m)$. 
   By Caratheodory's convex hull theorem \cite[Theorem I.2.3]{Bar02},
  $E_n(m)$ is a subset of $P_{n(\alpha+1)}(m)$ (where $\alpha$ is the dimension of 
   $\mathbb S_m$).  Since $P_n(m)$ is compact (being the image of the compact set
   $\{m\times n$ isometries$\}\times \Gamma(n)$ under the continuous map $(V,X)\mapsto
   V^*XV$), then so is $E_n(m)$.
  We have,
\[
  \mco \Gamma(m) = \bigcup_{n\geq m} P_n(m)= \bigcup_{n\geq m} E_n(m).
\]
  Thus, $\mco \Gamma(m)$ is the nested increasing union of a canonical sequence of compact convex sets.
\end{proof}

\subsection{Basic Definitions. Operator Level}
All the notions discussed above have  natural counterparts on infinite-dimensional Hilbert spaces.

  Fix a separable Hilbert space $\cK$ and let 
  $\LK$ denote the \df{lattice of subspaces of} $\cK$.  For a $K\in\LK$, 
 let $\smatkg$ denote $g$-tuples $X=(X_1,\dots,X_g)$
  of self-adjoint operators on $K$.
   A collection $\Gamma=(\Gamma(K))_K$ where $\Gamma(K)\subset \smatkg$ 
  for each $K\leq\cK,$ is a \df{free operator set}
  if it is closed under direct sums and with
  respect to simultaneous conjugation by unitary operators.
  If in addition it is closed with respect to simultaneous
 conjugation by  isometries $V:H\to K$, where $H, K\in\LK$,
 then $\Gamma$ is \df{operator convex}. 

  Note that $(\smatkg)_K$ is itself a 
  free operator set which will be henceforth denoted
  by $\smatog$.
 Given a symmetric free matrix polynomial $p$  with $p(0)\succ0$, let
\[
 \cO_p = \{X\in \smatog: p(X)\succeq 0\}
\]
be the \df{operator free semialgebraic set} defined by $p$.
  It is easy to see that $\cO_p$ is a free operator set. 
  For $K\in\LK$, we write 
  \[
  \cO_p(K)=\{
  X\in\smatkg: p(X)\succeq0
  \}.
  \]
  A  free operator semialgebraic set $\Gamma$ 
   is \df{uniformly bounded} if there is a $C\in\R_{>0}$ such that $C-\sum X_j^2 \succeq 0$
  for all $X\in\Gamma$.  

\subsection{The Operator Convex Hull}

Each free polynomial $p$ gives rise to 
 two operator convex hulls.
The \df{operator convex hull of $\cD_p$} is the sequence
  of sets $\mcop= (\mcop(n))_n$ where $X\in\smatng$ is in $\big(\mcop\big)(n)$ if there exists
  a $Z\in\cO_p$ (acting on a Hilbert space $\cK$) and an isometry $V:\FF^n\to\cK$
 such that $X= V^*ZV$. 

 The notion of the \df{(operator) convex hull of $\cO_p$} is defined similarly. Thus
   $\mcoop$ is the sequence of sets $(\mcoop(K))_K$ where, for $K\in\LK$, 
  the tuple  $X\in\smatkg$ 
   is in $\big(\mcoop\big)(K)$ if there exists
  a $Z\in\cO_p$ (acting on the Hilbert space $\cK$) and an isometry $V:K \to\cK$
 such that $X= V^*ZV$.

 Later we will see in Theorem \ref{thm:infinitelasserre}  that 
 $\mcop$ is closed.

\section{Linear Pencils and Matrix Convex Hulls}\label{sec:cp} 
Classical convex sets in $\R^g$ are defined 
as intersections of half-spaces and are thus described by linear functionals.
Matrix convex sets are defined analogously by linear pencils; cf.~\cite{EW97,HM12}.
This section surveys some basic facts about convex hulls and their associated linear pencils.

\subsection{Linear Pencils}
  Given $k\times k$ self-adjoint matrices $A_0,\dots,A_g$, let
\[
  L(x) = A_0 +\sum_{j=1}^g A_j x_j \in \mathbb S_k \langle x \rangle
\]
 denote the corresponding \df{(affine) linear pencil} of \df{size $k.$}
  In the case that $A_0=0$; i.e.,  $A=(A_1,...,A_g) \in \mathbb{S}_k^g$, let
  \[
     \La_A(x)= \sum_{j=1}^g A_j x_j
    \]
    denote the corresponding \df{homogeneous (truly) linear pencil}
    and
  \[
   L_A = I+\La_A
  \]
  the associated \df{monic linear pencil}. 

  The linear pencil can of course be evaluated at a point
 $x\in\mathbb R^g$ in the obvious way, producing the
  Linear Matrix Inequality, $L(x)\succeq 0$.  The solution
 set to this inequality is known as a 
 \df{spectrahedron} or \df{LMI domain} and is obviously a convex semialgebraic set. 
 
  The pencil $L$ is a free object too as it is naturally evaluated on $X \in \smatng$
using (Kronecker's) tensor product
\begin{equation}
 L(X):=  A_0 \otimes I + \sum_{j=1}^g A_j \otimes X_j.
\end{equation}
The free semialgebraic set $\cD_{L}$  
is easily seen to be matrix convex.  We will refer to 
 $\cD_{L}$ 
 as  a \df{free spectrahedron} or
 \df{free LMI domain}  and
say that a free set $\Gamma$ is \df{freely LMI representable}
 if there is a linear pencil $L$ such that $\Gamma =\cD_{L}$.
 In particular, if $\Gamma$ is freely LMI representable with a monic $L_A$, then $0$ is in the interior of $\Gamma$.
 Note too that $\cD_L(1)\subset\mathbb R^g$ is a spectrahedron.

Later we shall also use linear pencils which are based on infinite-dimensional 
operators $A_i$ and the associated pencil $L(x)=A_0+\sum A_j x_j$. In this case the
free set $\cD = (\cD(n))$, where $\cD(n) =\{X\in\smatng : L(X)\succeq 0\}$ is an 
\df{infinite spectrahedron}. We emphasize that the unmodified term free spectrahedron always
requires the $A_i$ to act on a finite-dimensional space.

The following is a special case (see  \cite[\S 6]{HM12}) of a Hahn-Banach separation theorem due to
Effros and Winkler \cite{EW97}. 

\begin{theorem}
 \label{prop:sharp}
  If $\cC=(\cC(n))_{n\in\N} \subseteq\smatg$ is a closed matrix convex set containing $0$ and
  $Y\in\smatmg$ is not in $\cC(m)$, then there
 is a monic linear pencil $L$ of size $m$ such that
  $L(X)\succeq 0$ for all  $X\in\cC$, but
  $L(Y)\not\succeq 0$. 
%
%
\end{theorem}

\begin{proof}
 From \cite[Theorem 5.4]{EW97}, there exist $m\times m$ matrices $A_1,\dots,A_g\in \FF^{m\times m}$ such 
  that
\[
 I - \frac12  \Big(\sum A_j \otimes X_j + \big(\sum A_j \otimes X_j\big)^*\Big) \succeq 0
\]
 for all $n$ and $X\in\cC(n)$, but at the same time
\[
 I - \frac12 \Big(\sum A_j \otimes Y_j + \big(\sum A_j \otimes Y_j\big)^*\Big) \not\succeq 0
\]
  Note however, that since $X_j^*=X_j$, it follows that $A_j\otimes X_j + A_j^* \otimes X_j^* = (A_j+A_j^*)\otimes X_j$.
  Thus, it can be assumed that $A\in \smatmg$.  
\end{proof}

Though linear matrix inequalities appear special, the following 
 result from \cite{HM12} says that they actually account for
  matrix convexity  of free semialgebraic sets.

\begin{theorem}
 \label{thm:HMlmirep}
  Fix $p$ a 
   symmetric  real matrix polynomial. 
   If 
   $p(0)\succ 0$  and  the strict positivity set 
   $\gP_p=\{X: p(X)\succ0\}$  of $p$ is bounded, 
then   $\gP_p$ 
       is matrix convex if and only if there is a monic linear pencil  $L$ 
 such that $\gP_p = \gP_L = \{X: L(X)\succ 0\}.$ 
\end{theorem}

\subsection{Pencils and Hulls}
 
\begin{lemma}
 \label{lem:sizemu}\rm
   Let $\cC$ be a matrix convex set. If $L$ is a pencil of size $k$, 
  then $L$ is positive semidefinite on $\cC$ if and only if $L$
  is positive semidefinite on $\cC(k)$. 
\end{lemma}

\begin{proof}
  Suppose $L$ is positive semidefinite on $\cC(k)$ and let $m$
  and $X\in\cC(m)$ be given. Fix a vector $v\in \FF^k \otimes \FF^m$.
  Letting $\{e_1,\dots,e_k\}$ denote the standard orthonormal basis for
 $\FF^k$, there exist vectors $v_1,\dots,v_k\in\FF^m$ such that
\[
 v =\sum_{j=1}^k e_j\otimes v_j.
\]
 Let $H$ denote the span of $\{v_1,\dots,v_k\}$ and let $V:H\to \FF^m$
 denote the inclusion mapping. It follows that
\[
 \begin{split}
 \langle L(X)v,v \rangle & =  \langle L(X) I\otimes V v, I\otimes V v\rangle\\
  &  =  \langle L(V^*XV) v,v \rangle.
 \end{split}
\]
 Since $V^* X V \in\cC(k)$, it follows that $\langle L(X)v,v\rangle \ge 0$.
 Hence $L(X)\succeq 0$ and the proof is complete.
\end{proof}

\begin{prop}\label{prop:useless}
Let $L$ be a $\mu\times\mu$ linear pencil. Then 
\[L|_{\cD_p}\succeq0 \quad\iff\quad L|_{\mco \cD_p(\mu)}\succeq0.\]
\end{prop}

\def\La{\Lambda}

Of course, the downside of Proposition \ref{prop:useless} is that it does not give
bounds on the isometries needed in $\mco \cD_p(\mu)$ (as they appear in Equation \eqref{eq:Cp}).

\begin{proof}
 Evidently $L$ is positive semidefinite on $\cD_p$ if and only if $L$ is
  positive semidefinite on $\mco\cD_p$. An application of
 Lemma \ref{lem:sizemu} completes the proof. 
\end{proof}

Just like the closed convex hull of a subset $C$ of $\R^g$ can be written as an intersection
of half-spaces containing $C$, closed matrix convex hulls are intersections of 
free spectrahedra.

\begin{cor}\label{cor:hb}
Let $p$ be a symmetric free polynomial $($with as usual $p(0)\succ 0)$.
For $n\in\N$, the set
$\cmco \cD_p(n)$ consists of all $g$-tuples $Z \in\smatng$ satisfying
$ L(Z)\succeq0$ for all $n\times n$ 
monic linear pencils $L$ 
with $\cD_L\supseteq\cD_p$ $($equivalently $L|_{\cD_p}\succeq0)$.
\end{cor}

\begin{proof} 
This corollary is a version of the matricial Hahn-Banach Theorem \ref{prop:sharp}.
Indeed, if $Z\not \in \cmco \cD_p(n),$ then by 
these matricial Hahn-Banach theorems 
there is an $n\times n$ pencil $L$ with $L(Z)\not\succeq0$ and $L|_{\mco \cD_p(n)}\succeq0$.
The latter implies by Proposition \ref{prop:useless} that $L|_{\cD_p}\succeq0$,
that is, $\cD_L \supset \cD_p$.

To prove the reverse inclusion,  suppose 
 $L$ is $n\times n$ with $L|_{\cD_p}\succeq0$.
If
  $Z\in\cmco \cD_p(n)$  
  there is a sequence    $Z_i \in\mco \cD_p(n)$  converging to $Z$.
Such $Z_i$ must have  the  form 
 $Z_i =V_i ^*X_i V_i $, with $X_i \in\cD_p$ and  $V_i $ is an isometry.
Thus
\[
L(V_i ^*X_i V_i )=(I \otimes V_i )^* L(X_i )(I\otimes V_i )\succeq0.
\]
Since $L(Z_i ) \to L(Z)$, we have
 $L \succeq 0$ on $ \cmco \cD_p(n)$, and so we are done. 
\end{proof}

\begin{cor}\label{cor:LMI1}
Suppose $\ell$ is an affine linear function, and $L$ is a linear pencil.
Then
\[
\ell|_{\cD_L(1)}\geq0 \quad\iff\quad \ell|_{\cD_L}\succeq0.
\]
\end{cor}

\begin{proof}
While  this is an obvious corollary of Proposition \ref{prop:useless}, let
us present a short and independent self-contained argument.
The implication $(\Leftarrow)$ is obvious.
For the converse assume  $X\in\cD_L(n)$ with $\ell(X)\not\succeq0$. Let
$v$ be a unit eigenvector of $\ell(X)$ with negative eigenvalue.
For 
$v^*Xv:=(v^*X_1v,\ldots,v^*X_gv)\in\R^g$ 
we have
$$L(v^*Xv)=(I\otimes v)^*L(X)(I\otimes v)\succeq0,$$
i.e., $v^*Xv\in \cD_L(1)$, and
\[\ell(v^*Xv)=v^*\ell(X)v<0.\qedhere\]
\end{proof}

\section{Projections of Free Spectrahedra: Free Spectrahedrops}\label{sec:4}

 Let $L$ be a linear pencil in the variables $(x_1,\dots,x_g;y_1,\dots,y_h)$,
\[
 L = D + \sum_{j=1}^g A_j x_j +\sum_{\ell=1}^h B_\ell y_\ell.
\]
 The set
\[
 \proj_x \cD_L(1) = \{x\in\mathbb R^g: \exists\, y\in\mathbb R^h \mbox{ such that }
    L(x,y) \succeq 0\}
\]
 is known as a \df{spectrahedral shadow} or is a \df{semidefinite programming (SDP) representable set} \cite{BPR13}
  and the representation afforded by $L$ is an \df{SDP representation}.  SDP representable sets are evidently
 convex and 
 lie in a middle ground between LMI representable sets and general convex sets. 
 They play an important role in convex optimization. 
 In the case that $S\subset \mathbb R^g$ is closed  semialgebraic and with some mild
 additional hypothesis, it is proved in \cite{HN10} based upon the Lasserre--Parrilo construction 
 (\cite{Las,Par06})
that the convex hull of $S$ is SDP representable.

  Given a linear pencil $L$, let $\proj_x \cD_L$ denote the free set
\[
 \proj_x \cD_L = \bigcup_{n\in\N} \{X\in\smatng : \exists\, Y\in\smatnh
  \mbox{ such that } L(X,Y)\succeq 0\}.
\]
  We will call a set of the form $\proj_x\cD_L$ a \df{free spectrahedrop} 
  or a \df{freely SDP representable set} or even a \df{free spectrahedral shadow}.
  Thus a free spectrahedrop is  a coordinate projection of a free spectrahedron.

\begin{prop}
 \label{prop:shadowclosedrestrict}
   Free spectrahedrops are matrix convex. In particular, they  are closed with respect to restrictions to reducing subspaces. 
\end{prop}

\begin{example}\rm
 \label{ex:badrestrict}
The second half of  Proposition \ref{prop:shadowclosedrestrict} fails for projections of general free semialgebraic sets.
  As an example, 
    consider
\begin{equation}
\label{eq:exNotRRR}
q=y x^2 y + z x^2 z - 1
\end{equation}
and the projection $\cS$ of  $\cD_{q}$ onto $x$.  Thus,
\[
 \cS =\{X\in\mbS: \exists \ (Y,Z)\in\mbS^2 \mbox{ such that } q(X,Y,Z)\succeq 0\}.
\]
It is easy to show $I_3 \oplus 0_3$
is in $\cS,$  but of course $0_3$  is not.
Incidentally, this gives a simple example of a free semialgebraic set
whose projection is not semialgebraic, in sharp contrast to 
Tarski's transfer principle in
classical real algebraic geometry  \cite{BCR98}.

On the other hand, Proposition \ref{prop:shadowclosedrestrict} implies that,
 for a  linear pencil $L$, projections of $\cD_L$ 
  are closed with respect to restrictions to reducing subspaces. 
Nevertheless, a projection of $\cD_L$ need not be semialgebraic,
cf.~\cite[\S 9]{HM12}.
\end{example}

\def\calK{\mathcal K}

\subsection{Free Spectrahedrops and Monic Lifts}
Recall a free set $\calK$ is a  free spectrahedrop 
if it is a (coordinate) projection of a free spectrahedron,
$\cD_\Lambda$.
The next lemma shows that even when $\Lambda$ is not
 a monic pencil, if $0$ is in the interior of $\calK,$ then $\calK$ 
admits a \emph{monic} LMI lift.

\begin{lem} 
\label{lem:boundedmonicLift}\rm
 If $\calK= \proj_x \cD_\La $ is a free spectrahedrop  containing $0$ in its interior,
 then there exists a monic 
 linear pencil $L(x,y)$ such that 
\begin{equation}\label{eq:lift-bounded}
\calK= \proj_x \cD_L = \{ X\in\smatg \mid \exists Y\in\smath: \, L(X,Y)\succeq0\}.
\end{equation}
If $\cD_\La$ is bounded, then we may further ensure $\cD_L$ is bounded.
\end{lem}

\begin{proof}
Suppose 
\begin{enumerate}[\rm (i)]
 \item $\La(x,y)$ is an affine linear pencil,
 \[
  \La(x,y) = \Lambda_0 + \sum_{j=1}^g \Lambda_j x_j + \sum_{k=1}^h \Omega_k y_k; 
 \]
 \item   $\calK$ is the projection of $\cD_{\La}$ onto  $x$-space. Thus,
$\calK=\{X\in\smatg: \exists Y\in\smath: \, \Lambda(X,Y)\succeq 0\}.$
\end{enumerate}
   Without loss of generality, it may be assumed the number $h$ of $y$-variables
  is the smallest possible with respect to the the properties (i) and (ii). 

Let $\ccD_\La(1)$ denote the interior of $\cD_\La(1)$ and suppose first that this interior is empty.
In this case  the convex subset
$\cD_\La(1)$ of $\mathbb R^{g+h}$  lies in a proper affine subspace of $\R^{g+h}$. That is, 
there is an affine linear functional (with real coefficients) 
\[
 \ell(x,y) =\ell_0 + \sum_{j=1}^g \ell_j x_j + \sum_{k=1}^h \omega_k y_k 
\]
 such that
$\ell=0$ on $\cD_\La(1)$. Equivalently, $\ell=0$ on $\cD_\La$, cf.~Corollary 
\ref{cor:LMI1}. At least one $\omega_k$ is nonzero
as otherwise
$\ell$ would produce a nontrivial affine linear map vanishing on $\calK$, contradicting
  the assumption that $\calK$ has nonempty  interior. Without loss of generality, $\omega_h=1.$ 
  Consider the pencil $\tilde{\Lambda}$ in the variables $(x,\tilde{y}) = (x_1,\dots,x_g,y_1,\dots,y_{h-1}),$ 
\[
  \tilde{\Lambda}(x,\tilde{y}) = \ell_0 (\Lambda_0-\Omega_h) +  \sum_{j=1}^g (\Lambda_j-\ell_j\Omega_h) x_j + \sum_{k=1}^{h-1} (\Omega_k - \omega_k \Omega_h)y_k
          = \Lambda(x,y) - \Omega_h \ell(x,y).
\]
 Given $X\in\calK(n)$, there is a $Y\in \smatnh$ such that $\Lambda(X,Y)\succeq 0$. Letting $\tilde{Y}=(Y_1,\dots,Y_{h-1})$,
\[
 \tilde{\Lambda}(X,\tilde{Y}) = \Lambda(X,Y)-\Omega_h\otimes \ell(X,Y) = \Lambda(X,Y)\succeq 0.
\]
 On the other hand, if there is a $\tilde{Y}=(Y_1,\dots,Y_{h-1})$ such that $\tilde{\Lambda}(X,\tilde{Y}) \succeq 0$, then
 with 
\[
  Y_h = -(\ell_0 I + \sum_{j=1}^g \ell_j X_j +  \sum_{k=1}^{h-1} \omega_k Y_k),
\]
 and $Y=(\tilde{Y},Y_h)$, it follows that $\ell(X,Y)=0$. Hence,
\[
  \Lambda (X,Y) = \tilde{\Lambda}(X,\tilde{Y}) + \Omega_h \otimes \ell(X,Y) = \tilde{\Lambda}(X,\tilde{Y})\succeq 0.
\]
 It follows that $\tilde{\Lambda}$ satisfies conditions (i) and (ii), contradicting
 the minimality assumption on the number of $y$-variables.  Hence $\cD_\La(1)$ has a nontrivial interior. 

The projection $\proj_x:\cD_\La(1)\to \calK(1)$ is continuous, so the preimage
of a small ball $B_\eps\subseteq\R^g$ around $0\in \inter \calK(1)$ is an open subset of $\cD_\La(1)$.
At least one of these points will have its $x$-component equal to $0$,
say $(0,\hat{y})\in\ccD_\La(1)$. By replacing $\La(x,y)$ with $L(x,y)=\La(x,y-\hat{y})$ 
we
obtain a linear pencil $L$ such that $\proj_x\cD_L=\proj_x\cD_\La$ 
but now the free spectrahedron $\cD_L$  has  $(0,0)$ as an  interior point.
Hence a standard reduction shows we may take $L$ to be monic
(cf.~\cite{HV07}).
It is clear that $\cD_L$ is bounded if $\cD_\La$ is bounded.
\end{proof}

\subsection{Convex Hulls and Spectrahedrops}
  Given a free semialgebraic set  $\cD_p$, a goal is to determine when its convex hull,
  or closed convex hull, or its operator convex hull is a free spectrahedrop.  
When this can be done it provides a potentially useful approximation to $\cD_p$.

\begin{example}
 \label{ex:btvlift}
 Recall the polynomial $p=1-x_1^2 -x_2^4$ from Example \ref{ex:btv}. 
  That the set $\cD_p(1)$ has an LMI lift is well known and is given 
 as follows. Let 
\[
  \Lambda(x_1,x_2,y) = \begin{pmatrix} 1  & 0 & x_1 \\ 0 & 1 & y \\ x_1 & y & 1 \end{pmatrix}
    \oplus \begin{pmatrix} 1 & x_2 \\ x_2 & y\end{pmatrix}.
\]
 It is readily checked that $\proj_x \cD_{\Lambda}(1) = \cD_p(1)$.  Further,
  Lemma \ref{lem:boundedmonicLift} implies that $\Lambda$ can 
  be replaced by a monic linear pencil $L$, cf.~Subsection \ref{subsec:71}.
\end{example}

\begin{prop}\label{prop:liftMe}
 Assume $\cD_p(1)$ is bounded and $L$ is a monic linear pencil. 
 If  $\conv (\cD_p(1))$, the ordinary
 convex hull of $\cD_p(1)\subset \mathbb R^g$,  admits an LMI lift 
 to $\cD_L(1)$  and 
  $\cD_p\subseteq\proj_x \cD_L,$ then $\conv (\cD_p(1))=\mco \cD_p(1)$.
\end{prop}

\begin{proof}
Suppose $\ell$ is an affine linear function nonnegative on $\cD_p(1)$.
Then
 $\ell|_{\cD_L(1)}\geq0$ and hence by Corollary \ref{cor:LMI1},
 $\ell|_{\cD_L}\succeq0$. Since $\mco \cD_p 
\subseteq\proj_x \cD_L$,
 this implies $\ell|_{\mco \cD_p(1)}\geq0$. As $\cD_p(1)\subseteq\mco \cD_p(1)$,
 this shows $\overline {\conv \cD_p(1)}=\mco \cD_p(1)$.
 As $\cD_p(1)$ is compact, its convex hull is closed, so we are done.
\end{proof}

\begin{remark}\rm
 \label{rem:liftMe}
 Later in Section \ref{sec:freelassy}  we shall give a procedure 
 for constructing a family of $L$
 with  the property 
 \beq\label{eq:subtle}
 \mco \cD_p\subseteq\proj_x \cD_L.
 \eeq
While for many $p$ 
the ordinary
 convex hull of $\cD_p(1)$,  admits an LMI lift 
 to $\cD_L(1)$, the property \eqref{eq:subtle} is  not always satisfied. Indeed,
the conclusion of Proposition \ref{prop:liftMe}
can fail. 
\end{remark}

\begin{example}\rm
 Returning to the polynomial $p(x_1,x_2)=1-x_1^2 - x_2^4$ of Example \ref{ex:btvlift},  Proposition \ref{prop:liftMe} 
  implies that $\cD_p(1) = \mco \cD_p(1)$.
 Since, as noted in Example \ref{ex:btvv},  $\cD_p(2)$ is not convex, $\cD_p$  is not a free spectrahedrop.
  We do not know if the  closed matrix 
  convex hull of $\cD_p$ is a free spectrahedrop, but Theorem \ref{thm:infinitelasserre} below
   says it almost is. 
\end{example}

\section{Construction of the Free Lift}
 \label{sec:freelassy}

Classically, given a  commutative semialgebraic set $\cD_p(1) \subseteq \R^g$,
 a  construction proposed by Lasserre \cite{Las} (see also Parrilo \cite{Par06})
produces a sequence of spectrahedra $(D^{(n)})$ and projections  $(\pi_n)$
 such that $\pi_n(D^{(n)} )\supseteq \cD_p(1)$ is a nested decreasing sequence of semialgebraic sets
approximating the convex hull of $\cD_p(1)$.  
 Under mild hypotheses, this {\em sequence of  relaxations}
  actually terminates and  presents
  $\conv \big(\cD_p(1)\big)$ as a projection of a spectrahedron; i.e., there is an $m$ 
 such that $\conv \big(\cD_p(1)\big)=\pi_m(\cD^{(m)})$  \cite{HN09,HN10}.
 For a substantial recent advance, see Scheiderer's
 complete solution in two dimensions \cite{Sch11, SchArx12}.
 We refer to \cite{dKL11,Hen11,NPS10} for further results in this direction.
 
There are two parts to our
 free analog of the Lasserre--Parrilo construction. The first, described in this section,
 constructs for a given $\cD_p$, via free analogs of moment sequences and Hankel matrices, 
an infinite  free  spectrahedron $\cL_p,$ and a canonical projection  of $\cL_p$ onto 
   the operator convex hull of $\cD_p$. 

   The second part of the construction, appearing in Section \ref{sec:preEx}, 
   consists of a systematic procedure for passing from 
  $\cL_p$ to a sequence of finite free spectrahedra and corresponding projections
  onto increasingly finer outer approximations to the operator convex hull of $\cD_p$.

\subsection{Free Hankel matrices} 
The key ingredient of the systematic method for constructing lifts presented here
are the block free (multivariable) analogs of Hankel matrices. A Hankel matrix $H$
is one that is constant on antidiagonals so that the entry $H_{i,j}$ depends only on the sum $i+j$.
 In particular, a sequence $(h_k)_k$ of self-adjoint $m\times m$ matrices 
 determines a block Hankel matrix $H= (h_{i+j})_{i,j}$.
 The sequence $(h_k)$ is often referred to as a moment sequence.
  In the case that $H$ is positive semidefinite, the normalization $h_0=I$
  is typically harmless. 

 Free Hankel matrices have a description in terms of free moment sequences. 
 Given a positive integer  $n,$  a sequence  $W:= (W_\alpha)_\alpha$ of 
 $n\times n$ matrices $W_\alpha$ indexed by words $\alpha$
 in the free symmetric variables $x=(x_1,\dots,x_g)$ 
 is a \df{moment sequence} if it 
 is symmetric in the sense that $W_{\alpha^*} = W^*_\alpha$
  and is normalized by $W_\emptyset =I$. 
  Note that the symmetry of $W$ implies that each $W_{x_j}$ is 
  a self-adjoint matrix.
  The moment sequence $(W_\alpha)$ determines 
  the \df{free Hankel} matrix 
\[
  H(W)= \big(   W_{\alpha^*  \beta}  \big)_{\alpha,\beta}.
\]
 For a  positive integer $d$, 
$$
  H_d(W)= \big( W_{\alpha^*  \beta} \big)_{|\alpha|, |\beta|\le d}.
$$
  is a \df{truncated free Hankel matrix} associated to $W$.

 Let $p$ be a  $\ll\times\ll$-matrix valued polynomial of degree 
 at most $\delta$. Thus, 
\[
  p(x) = \sum_{ | \gamma| \leq \de} p_\gamma \gamma
\]
 for some $\ll\times\ll$ matrices $p_{\gamma}$. 
 The  \df{$p-$localizing matrix} $\hs_p (W)$ associated to $H(W)$ 
 is the $n\ll\times n\ll$ (block) matrix with $(\alpha,\beta)$ entry 
$$
 \hs_p (W)_{\alpha,\beta}: = \sum_{|\gamma| \le \de}
 p_\gamma  \otimes W_{\alpha^* \gamma \beta}.
 $$
  Of course, if $p=1$, then 
$$
 \hs_1 (W) = H (W).
$$
For $d\in\N$, the \df{$d$-truncated localizing matrix} of $p$ is 
\[
\hs_{p,d}(W):=
\big(\hs_p (W)_{\alpha,\beta}\big)_{|\alpha|,|\beta|\leq d}.
\]
 Note that if the word $\gamma$ has length $2m-1$ or $2m$, then it can be written
 as a product $\gamma = \eta^*\sigma$ of words of length at most $m$. 
 Hence, the truncated localizing matrix actually only depends upon 
 the entries $W_{\al^*\be}$ for
$|\al|,|\be|\leq d+\left\lceil \frac12\deg(p)\right\rceil$.
Here $\lceil\textvisiblespace\rceil$ denotes the 
``smallest integer not less than'' function.
 The reader is encouraged to skip ahead temporarily 
 to Subsection \ref{subsec:hankelFun} to get a feel for
 the structure of these matrices. 

\bigskip

  An element $Z\in \cD_p(m)$ (so acting on $\FF^m$) along with
   an isometry $V:\FF^n \to \FF^m$ determines a moment sequence,
\begin{equation}
 \label{eq:moment}
  Y_\alpha = V^* Z^\alpha V. 
\end{equation}
   For instance, if $\alpha =x_1x_2x_1$, then
\[
 Y_\alpha = V^* Z_1 Z_2 Z_1 V.
\]
  Note that the fact that $Z^\emptyset=I$ and the assumption that $V$ is an isometry
  implies $Y_\emptyset =I$.  Further,
 an easy calculation shows that this moment sequence satisfies
 \begin{equation}\label{eq:momentPos}
H(Y) \succeq 0 \quad\text{ and }\quad
\hs_p(Y) \succeq 0.  
\end{equation}
  Likewise, an element $Z\in\cO_p(K)$ along with $n\in\N$ and an isometry 
 $V:\FF^n\to K$ determines a moment sequence $(Y_\al)_\al$ via \eqref{eq:moment} for which
 \eqref{eq:momentPos} holds.

\subsection{Riesz Maps}
 \label{sec:Riesz}
\def\Rss{\R^{s\times s}}
Let $s\in\N$.
To a moment sequence $W=(W_\al)_\al$ of $n\times n$ matrices there is the associated 
linear \df{Riesz mapping}
\[
\Phi_W^s: \FF^{s\times s}\ax\to   \FF^{sn \times sn}, 
\quad
\sum_{\al\in\ax} B_\al \al \mapsto\sum_{\al\in\ax}  B_{\al}\otimes W_{\al}.
\]
This linear map is symmetric in the sense that 
\[
\Phi_W^s(P^*)=\Phi_W^s(P)^*
\]
for $P\in\FF^{s\times s}\ax$.

Similarly, to a truncated Hankel matrix
$H_d(W)$, or the corresponding truncated moment sequence $W=(W_\al)_{|\al|\leq 2d}$,
we can associate a \df{Riesz map}
\[
\Phi_W^s: \FF^{s\times s}\ax_{2d}\to      \FF^{sn \times sn},        \quad
\sum_{\al\in\ax_{2d}} B_\al \al \mapsto \sum_{\al\in\ax_{2d}} B_{\al}\otimes W_{\al}.
\]

\begin{prop}\label{prop:riesz}
Suppose $W$ is a moment sequence and let $p\in\FF^{\ell\times\ell}\ax$ be a symmetric free matrix polynomial.
 For positive integers $s$ and $t$, 
\ben[\rm (1)]
\item
if $H(W)\succeq0,$ then $\Phi_W^s(P^*P)\succeq0$ for all $P\in\FF^{t\times s}\ax$;
\item
if $H_d(W)\succeq0,$ then $\Phi_W^s(P^*P)\succeq0$ for all $P\in\FF^{t\times s}\ax_d$;
\item
if $\hs_p(W)\succeq0,$ then $\Phi_W^s(f^* (I_t\otimes p) f)\succeq0$ for all $f\in\FF^{t\ell\times s}\ax$;
\item
if $\hs_{p,d}(W)\succeq0,$ then $\Phi_W^s(f^* (I_t\otimes p) f)\succeq0$ for all $f\in\FF^{t\ell\times s}\ax_d$.
\een
\end{prop}

\begin{proof}
(1) Write $P=\sum_{\al\in\ax} P_\al \al$. Then
\[
\Phi_W^s(P^*P) = \sum_{\al,\be} P_\al^* P_{\be} \otimes W_{\al^*\be}.
\]
Let $\vec P = \begin{pmatrix} P_v \end{pmatrix}_{v\in\ax}$ be a column block-vector of coefficients of $P$. Then
\beq\label{eq:riesz1}
\Phi_W^s(P^*P) = \big(\vec P \otimes I_n\big)^* \big(I_t\otimes H(W) \big) \big(\vec P \otimes I_n\big)
\succeq0
\eeq
since $I_t\otimes H(W)\succeq0$ by assumption.
For the proof of (2) simply replace $H(W)$ by $H_d(W)$ in \eqref{eq:riesz1}.

The proofs of (3) and (4) are similar to those of (1) and (2) respectively. 
For (3), using the vector notation as in items (1) and (2), 
\[
\begin{split}
 \Phi_W^s\big(f^*(I_t\otimes p) f\big) & =  \sum_{\sigma}  \Big(\sum_{\alpha,\beta}  \sum_{\gamma: \alpha^*\gamma\beta=\sigma} f_\alpha^* p_\gamma f_{\beta}\Big) \otimes W_{\sigma} \\
&  =  \sum_{\alpha,\beta} \big(f_\alpha^*\otimes I_n\big)\, \Big(\sum_{\gamma} I_t\otimes p_{\gamma}\otimes W_{\alpha^*\gamma\beta} \Big)\,  \big(f_\beta \otimes I_n\big) \\
&  =  \big(\vec f \otimes I_n\big)^*  \big(I_t\otimes \hs_p(W)\big) \big(\vec f \otimes I_n\big).
\end{split}
\]
 For (4) we use the corresponding truncated version
\[
\Phi_W^s(f^*p f) = \big(\vec f \otimes I_n\big)^*  \big(I_t\otimes \hs_{p,d}(W)\big) \big(\vec f \otimes I_n\big),
\]
where $\vec f$ is a block column vector consisting of coefficients of $f$.
\end{proof}

\subsection{Lasserre--Parrilo Lift: Moment Relaxations}
 \label{sec:LasLiftDef}
 Given a positive integer $n$, let
\[
 \Lift_p(n) := \{Y=(Y_\alpha)_\alpha : Y_\alpha \in \FFnn, \ \ Y_\emptyset =I, 
 \ \ Y_{\al^*}=Y_\al^*, 
 \ \ H(Y) \succeq 0, \ \ \hs_p(Y) \succeq 0\}
\]
 and let $\Lift_p$ denote the sequence $(\Lift_p(n))_n$.  Implicitly, the $Y$ in $\Lift_p$
  are understood to be moment sequences. 
 Moreover, let
 \[
 \Liftfin_p:=\{ Y\in \Lift_p : \rank H(Y)<\infty\}.
 \]
 In particular, the $Y$ appearing in  \eqref{eq:moment} is in $\Lift_p$ if $Z\in\cO_p,$ and is in $\Liftfin_p$ if $Z\in\cD_p$,. 
 Given $Y\in\Lift_p(n)$, let 
\[
  \hY = (Y_{x_1}, Y_{x_2},\dots, Y_{x_g}) \in \smatng.
\]

\begin{thm}
 \label{thm:true}
   If $X\in\mco \cD_p$,
   then there is a $Y\in\Liftfin_p$ such that
\[
  X=\hY.
\]
  Conversely, if  $Y\in\Liftfin_p$, then $\hY\in\mco \cD_p$. 
  \end{thm}

\begin{proof}
 If $X$ is in the matrix convex hull of $\cD_p$, then there is an isometry $Q$
 and $Z\in\cD_p$ such that  $X=Q^* ZQ$. In this case the moment sequence
  $Y_\alpha = Q^* Z^\alpha Q$ satisfies the conclusion of the first part of
  the theorem.

  To prove the converse, suppose $(Y_\alpha)$ is a moment sequence from
  $\Liftfin_p(n)$.  
 Define, on the vector space $\FF\langle x\rangle \otimes \FF^n$, the 
  sesquilinear form
\begin{equation}
 \label{eq:hankform}
 [ s,t ]_Y = \sum_{\alpha,\beta}  \langle Y_{\beta^* \alpha} s_\alpha, t_\beta \rangle
\end{equation}
 where $s=\sum \alpha \otimes s_\alpha$ and $t=\sum \beta\otimes t_\beta$. 
 The assumption that $H(Y)\succeq 0$ implies that the form $[s,t]_Y$ 
  is positive semidefinite. Let $\cE_Y$ denote the (pre-)Hilbert space 
  obtained by modding out the subspace 
\[
   \cN = \{s : [s,s]_Y = 0\}.
\]
 Note that $\rank H(Y)<\infty$ implies $\dim \cE_Y<\infty$ and hence $\cE_Y$ is a Hilbert space. 

 An important observation is the following: If $s\in \cN$ and $1\le j\le g$, then 
 $r=x_j s \in\cN,$
 i.e., $\cN$ is a left $\FF\ax$-submodule of $\FF\langle x\rangle \otimes \FF^n$.
    To prove this observation, note that, because
  $H(Y)$ is positive semidefinite, if $s\in\cN$ then
\[
  \sum_\alpha Y_{\beta^* \alpha} s_\alpha =0
\]
 for each $\beta$ (and conversely). In this case,
\[
  \sum_\gamma Y_{\beta^* \gamma} r_\gamma 
   =  \sum_{\alpha} Y_{\beta^* x_j\alpha} s_\alpha 
   =  \sum_{\gamma} Y_{(x_j\beta)^* \alpha} s_\alpha 
   =  0
\]
  and hence $r\in\cN$.  It now follows that the mapping $Z_j$
  sending $s$ to $x_j s$ is well defined on the
  finite-dimensional
  Hilbert space $\cE_Y$.  The computation above also shows that
  whether or not $s\in \mathcal N$, 
\[
  \langle x_j s,t\rangle = \langle s,x_jt\rangle
\]
 and hence $Z_j^*=Z_j$. 

Define $Q:\FF^n \to \cE_Y$ by
\[
 Qv = \emptyset \otimes v
\]
 and note that $Q$ is an isometry.  By construction, $Q^* Z^\alpha Q= Y_\alpha$.

 Finally, to see $p(Z)=\sum p_\gamma \otimes Z^\gamma$ is positive definite, let
 $s=\sum e_j\otimes \alpha \otimes s_{\alpha,j}$ be given, where $\{e_1,\dots,e_\ell\}$
  is the standard orthonormal basis for $\FF^\ell$ (the space that the $p_\gamma$ act on)
  and $s_{\alpha,j}\in\FF^n$. Then,
\[
 \begin{split}
   \langle p(Z)s,s \rangle & =  \sum_{\alpha,\beta,\gamma,j,k} \langle p_\gamma \otimes Z^\gamma e_j\otimes \alpha\otimes s_{\alpha,j}, e_k\otimes \beta \otimes s_{\beta,k}\rangle \\
&   =  \sum \langle p_\gamma e_j,e_k\rangle \, \langle Z^\gamma \alpha\otimes s_{\alpha,j}, \beta\otimes s_{\beta,k}\rangle \\
  & =  \sum \langle p_\gamma e_j,e_k\rangle \, \langle Y_{\beta^* \gamma \alpha} s_{\alpha,j},s_{\beta,k}\rangle \\
  & =  \sum_{\alpha,\beta}  \Big\langle \big(\sum_\gamma p_\gamma \otimes Y_{\beta^*\gamma\alpha}\big) \sum_j e_j \otimes s_{\alpha,j}, \sum_k e_k \otimes s_{\beta,k}\Big\rangle \\
   & =  \langle \hs_p(Y) \vec s,\vec s \, \rangle \ge 0,
\end{split}
\]
  where $\vec s$ is the vector $(s_\alpha)_\alpha$ for $s_\alpha = \sum_j e_j \otimes s_{\alpha,j}$.  Thus the assumption that $\hs_p(Y)$ is positive semidefinite
 implies $p(Z)$ is positive semidefinite. 
  We conclude that 
  $\hY = Q^* Z Q$ is in the matrix convex hull of $\cD_p$.
\end{proof}

\begin{definition}\rm
Given $p$, let 
\[
 \hLift_p := \{\hY: Y\in\Lift_p\} \quad\text{ and }\quad
 \hLiftfin_p := \{\hY: Y\in\Liftfin_p\}.
\]
\end{definition}
\noindent
Theorem  \ref{thm:true} says that the matrix
 convex hull $\mco \cD_p$ of $\cD_p$ equals  $\hLiftfin_p$.

  Next we turn to operator convex hulls. 
To obtain a good lifting theorem we make a boundedness assumption
which replaces the $\rank H(Y)$  finite condition used in Theorem \ref{thm:true}.\looseness=-1

  Given $K\in \R_{>0},$ the matrix polynomial $p$ is \df{$K$-archimedean} if
   there exist  matrix polynomials  $s_j$ and $f_j$ such that 
   \begin{equation}\label{eq:archimed}
  K^2 - \sum_j x_j^2  = \sum s_j^* s_j + \sum f_j^* p f_j,
\end{equation}
  and $p$ is \df{archimedean} if it is $K$-archimedean for some $K>0$.

\begin{thm}
 \label{thm:infinitelasserre}
   If $p$ is archimedean, then $\mcop = \hLift_p$. 
   Moreover, $\mcop(n)$ is closed and bounded and contains $\mco \cD_p(n)$
   for each $n$. 
\end{thm}

\subsubsection{Proof of Theorem {\rm\ref{thm:infinitelasserre}}}
 The proof begins with several lemmas.

\begin{lemma}
 \label{lem:arch-bounded}\rm
  If  $p$ is archimedean, then $\cO_p$ is uniformly bounded. 
\end{lemma}

\begin{proof}
If $p$ is archimedean, then by \eqref{eq:archimed}
there is $N\in\N$ with $N-\sum_j x_j^2 |_{\cO_p}\succeq0$.
Hence $\cO_p \subseteq \big\{ X\in\smatog : \|X\|^2\leq N\big\}.$
\end{proof}

\begin{lemma}
 \label{lem:basiclasserre}\rm
   If $Y\in \Lift_p(n)$, then there exist
 \begin{enumerate}[\rm (i)]
  \item  a Hilbert space $\cH$;
  \item  a dense subset $\cP$ of $\cH;$
  \item  a tuple $Z=(Z_1,\dots,Z_g)$    such that each $Z_j:\cP\to \cP$
     is self-adjoint in the sense that $\langle Z_j p,q\rangle = \langle p,Z_jq\rangle$
    for each $p,q\in \cP$; and
  \item  an isometry $V:\FF^n\to \cP$
\end{enumerate}
  such that
 \begin{enumerate}[\rm (a)]
   \item $p(Z):\cP\to \cP$ is positive semidefinite; 
   \item 
    \label{it:inopCp0}
        $\hat{Y} = V^* ZV$; and
        \item\label{it:inopCp}
        if in addition $p$ is $K$-archimedean, then 
  each $Z_j$ is a bounded operator $($and so extends to all of $\cH)$ with
 $K^2-\sum Z_j^2 \succeq 0$. 
   Hence the $\hat Y$ from \eqref{it:inopCp0} is in $\mcop$. 
 \end{enumerate}
\end{lemma}

\begin{proof}
  Following the proof of Theorem \ref{thm:true}, given 
  a moment sequence $(Y_\alpha) \in \Lift_p(n)$, define the 
  pre-inner product $[\tvs,\, \tvs]$ on
  $\mathscr{R}=\FF\langle x\rangle \otimes \FF^n$,
  as in Equation \eqref{eq:hankform}. 
   Let 
\[
 \cN = \{f\in\FF\langle x \rangle \otimes \FF^{n} : [f,f]=0\}.
\]
 A standard argument shows that $\cN$ is a subspace of $\mathscr{R}$
  and that the form 
\[
  [ f, g]=[f+\cN,g+\cN]
\]
 is well defined and positive definite
 on the  quotient $\cP$ of $\mathscr{R}$ by $\cN$. 

   The operators $Z_j$ of multiplication
  by $x_j$ are as before 
  (see the proof of Theorem \ref{thm:true})
  well defined relative to this pre-inner product;
  i.e., each $Z_j:\cP\to \cP$. 
  Moreover, 
  \begin{equation}\label{eq:pz}
  p(Z)\succeq 0
  \end{equation}
  on $\cP$ too.  Define $V:\FF^n \to \mathscr{R}$ by 
\[
  Vh = (\emptyset \otimes h)+\cN.
\]
  Then $V$ is an isometry (since $Y_{\emptyset}=I$) and $V^*ZV =\hat{Y}$.

  Let us show that the $Z_j$ are bounded 
  under the archimedean hypothesis.  By $K$-archimedeanity of $p$,
  \[
  K^2-\sum_j x_j^2 = \sum_i f_i^*f_i + \sum_k r_k^* p r_k
  \]
  for some free polynomials $f_i,r_k$. It is now clear that \eqref{eq:pz}
  implies $K^2-\sum_j Z_j^2\succeq0$, i.e., $\|Z\|^2\leq K^2$ so $Z$ is bounded. Then
  by definition, $\hat{Y}\in \mcop(n)$. 
\end{proof}

 The proof of the moreover statement in Theorem \ref{thm:infinitelasserre}
  will use the following lemma.

\begin{lemma}
 \label{lem:archtobound}\rm
   If $p$ is archimedean, then for each $\alpha$ there is 
   a constant $C_\alpha$ such that if $Y\in\Lift_p$, then
  $\|Y_\alpha\|\le C_\alpha$. Further, 
   if $(Y^j)_j=((Y_\alpha^j)_\alpha)_j$ is a sequence from $\Lift_p(n)$ satisfying for each $\alpha$ there
 is a $Y_\alpha$ such that $(Y_\alpha^j)_j$ converges to $Y_\alpha$, then
  $Y=(Y_\alpha)_\alpha \in\Lift_p(n)$.
\end{lemma}

\begin{proof}
  Suppose $p$ is $K$-archimedean.  
  By Lemma \ref{lem:basiclasserre}, given $Y\in\Lift_p(n)$
  there exists an operator tuple $Z$
  acting on a Hilbert space $\cH$  with $K^2-\sum Z_j^2\succeq 0$
  and $p(Z)\succeq 0$ as well as
  an isometry $V:\FF^n\to \cH$ such that
   $Y_\alpha = V^*Z^\alpha V$. Letting $|\alpha|$ denote the length
  of the word $\alpha$, it is immediate that
\[
  \|Y_\alpha\| \le K^{|\alpha|}. 
\]

  For the second part of the lemma, note that for $d$ fixed, each
 $H_d(Y^j)$ is positive semidefinite.  Since $H_d(Y^j)$ is a (finite) matrix
 and depends only upon $|\alpha|\le 2d$, it follows that 
 $(H_d(Y^j))_j$ converges to $H_d(Y)$. Thus $H_d(Y)$ is positive semidefinite.
  Since $H_d(Y)$ is positive semidefinite for all $d$, it follows
  that $H(Y)$ is also positive semidefinite.

  In a similar manner, each $\hs_{p,d}(Y^j)$ is positive semidefinite and,
 for $d$ fixed, $(\hs_{p,d}(Y^j))_j$ converges to $\hs_{p,d}(Y)$ and thus
  $\hs_{p,d}(Y)$ is positive semidefinite. It follows that
  $\hs_p(Y)$ is positive semidefinite. Thus $Y\in\Lift_p$.  
\end{proof}

\begin{proof}[Proof of Theorem {\rm\ref{thm:infinitelasserre}}]
 From Lemma \ref{lem:basiclasserre} it follows 
  that $\hLift_p\subset \mcop$.   The reverse inclusion
  follows along the lines of the proof of the similar
 statement in Theorem \ref{thm:true}. Namely, 
  simply observe that if $Z\in\cO_p$ and $V$
 is an isometry from $\FF^n$ into the space $\cH$ that
  $Z$ acts on, then $Y_\alpha = V^* Z^\alpha V$ defines
  a moment sequence $Y\in\Lift_p(n)$. 

  The inclusion $\cD_p \subset \mco\cD_p$ is evident. Likewise the
 archimedean hypothesis and Lemma \ref{lem:arch-bounded} readily imply
  the boundedness of $\mcop$. Thus, to finish the proof of the
 moreover statement, it remains to show each $\mcop(n)$ is closed.
  Accordingly, suppose $(X^{(j)})$ is a sequence from $\mco \cD_p(n)$ which converges
  to some $X\in\smatng$.  In particular, the $X^{(j)}$ act on $\FF^n$
  and for each $k$, the sequence $(X^{(j)}_k)_j$ converges to $X_k$. 
 
 For each $j$ there is a tuple $Z^{(j)}\in\cO_p$ acting on a Hilbert space $K_j$ 
  and an isometry $V_j:\FF^n\to K_j$ such that 
\[
  X^{(j)} = V_j^* Z^{(j)} V_j.
\]
  The moment sequence, $(Y_\alpha^{(j)})$ coming from the pairs $(Z^{(j)},V_j)$  via
\[
  Y_\alpha^{(j)} = V_j^* (Z^{(j)})^\alpha V_j
\]
 is in $\Lift_p$.

 For fixed $\alpha$, the hypothesis and Lemma \ref{lem:archtobound} together  imply
 that the sequence $(Y_\alpha^{(j)})$ is bounded 
  and thus has a convergent subsequence. 
  Thus, by passing to a subsequence
  (using the usual diagonalization argument) we can assume
 that, for each $\alpha,$ there is a $Y_\alpha$
 to which  $Y_\alpha^{(j)}$  converges. 
  By the second part of Lemma \ref{lem:archtobound},
  this moment sequence $(Y_\alpha)$
  belongs to $\Lift_p$. Hence the corresponding operator $Z$ 
  from Lemma \ref{lem:basiclasserre} satisfies
  $p(Z)\succeq 0.$ Thus $Z\in\cO_p$. By construction,
\[
  V_0^* Z V_0 = (Y_{x_1},\dots,Y_{x_g}) = X.
\]
 Hence $X\in \mcop(n)$ and therefore $\mcop(n)$ is closed.
\end{proof}

\begin{remark}\rm
 Note that the reverse inclusion, \[\mcop(n) \subset \cmco \cD_p(n),\]  holds exactly when 
 matrices - and not operators - suffice in the
 \cite{HM04a} Positivstellensatz.   Indeed, $\mcop$
 is the intersection of all $\cD_L$ for monic $L$
 such that $L(Z)\succeq 0$ for all $Z\in\cO_p$.
 On the other hand, $\cmco\cD_p$
  is the intersection of all $\cD_L$ for monic $L$
 such that $L(Z)\succeq 0$ for all $Z\in\cD_p$.
This theme was discussed in Corollary \ref{cor:hb} above; see also Subsection \ref{subsec:LPP}.
\end{remark}

\section{Truncated Moments - Approximations of the Matrix Convex Hull}
\label{sec:preEx}
  This section presents the second part of the  Lasserre--Parrilo construction 
  in the free setting. It consists of a sequence of truncations of the lift $\cL_p$
  of $\cD_p$    from Section \ref{sec:freelassy} to a sequence of finite free spectrahedra and
  corresponding projections onto increasingly finer outer approximations to the operator convex hull of $\cD_p$.
  Alternately, the construction can be thought of as producing a sequence
  of approximate free spectrahedral lifts of a given free  semialgebraic set $\cD_p$
  to LMI domains in increasingly many variables.

  Whether  this construction produces the matrix convex hull at a finite stage is a basic question. 
 In Subsection \ref{subsec:74} we give  examples where the answer is yes.  In fact, in these 
 the convex hulls involved require
 no lifts, they are themselves free spectrahedra. 
 In the other direction, for the 
TV screen
 the construction does not produce the matrix convex hull
  at  the first  stage, cf.~Section \ref{sec:ex}.

\subsection{Main  Formulas for Lifts}
\label{subsec:mainResults}

To state the  main result of this paper  precisely, 
for  $n\in\N$ and $d\in\N_0$, let
\begin{multline*}
  \Lift_p(n;d) : = \big\{ (Y_\al)_\al : |\alpha|\le 2d+\deg p+1
  ,  \ Y_\al \in \FF^{n\times n},  \ Y_\emptyset =I, \ Y_{\al^*}=Y_\al^*,  \\
  H_{d + \left\lceil\frac12\deg p\right\rceil} (Y) \succeq 0, \ \hs_{p,d}(Y)  \succeq  0 \big\}.
\end{multline*}
The sequence $\Lift_p(\textvisiblespace;d)=(\Lift_p(n;d))_n$ is a free convex set and, 
as before,  $\hLift_p(n;d)$ denotes the image of the projection
\[
  \Lift_p(n;d) \ni Y \mapsto \hY = (Y_{x_1},\dots,Y_{x_g}) \in\smatng.
\]

\begin{thm}[Clamping down theorem]
\label{thm:intersect}
 If $p$ is archimedean, then for each $n$,
\[
  \bigcap_{d=0}^\infty \hLift_p(n;d) = \hLift_p(n).
\]
\end{thm}

\begin{cor}
 \label{cor:seqlasserre}
   If $p$ is archimedean, then $\mcop(n) = \hLift_p(n)$ for each $n\in\N$. 
   Hence the sets $\hLift_p(n;d) $ close down on $\mcop(n)$.
   Further, for each $d$ there exists a  linear pencil $L_d$ such that
   $\hLift_p(\textvisiblespace;d)$ lifts to $\cD_{L_d}.$ Thus $\hLift_p(\textvisiblespace;d)$
   is a sequence of free spectrahedrops which are outer approximations to $\mcop$ 
   and which converge monotonically to $\mcop$ as $d$ tends to infinity.
\end{cor}

Corollary \ref{cor:seqlasserre} is an immediate consequence of Theorems  \ref{thm:intersect} and \ref{thm:infinitelasserre},
  and the fact that there exists a linear pencil  $L_d$ such that $\proj_x\cD_{L_d}=\hLift_p(\textvisiblespace;d)$ as we now explain.

\subsection{\texorpdfstring{$\hLift_p(\textvisiblespace ;d)$}{L--P lifts} are free spectrahedrops}

The free Lasserre--Parrilo construction produces the  approximate lifts 
  $\cD_\Delta=\Lift_p(\textvisiblespace;d)$  of $\mcop$,
 in which  $\cD_\Delta$ is the positivity  set of a linear matrix polynomial
\beq\label{eq:mixedUp}
  \Delta(x,y) = A_0 +\sum_{j=1}^g A_j x_j 
+ \sum_{\ell=1}^h \big( B_\ell y_\ell + B_\ell^* y_\ell^*)  
 \eeq
 where the coefficients are
 $k\times k$ self-adjoint matrices $A_0,\dots,A_g\in\mbS_k$, and $k\times k$ matrices
$B_1,\ldots,B_h\in\FF^{k\times k}$.
 This $\Delta$ can be naturally evaluated at tuples
 $(X,Y)\in \mbS_n^g \times (\FF^{n\times n})^h$ by
 \[
\Delta(X,Y) = A_0 \otimes I_n +\sum_{j=1}^g A_j \otimes X_j 
+ \sum_{\ell=1}^h \big( B_\ell \otimes Y_\ell + B_\ell^* \otimes Y_\ell^*)  
    \in \sar {nk}.
    \]
 While the coefficients  $A_j$ are self-adjoint and
the variables $x_j$ 
 are symmetric, 
  the coefficients and variables $B_\ell$ and  $y_\ell$ are not. Hence $\Delta$  is not a 
  linear pencil  according to the terminology in this article.
  The following lemma shows that 
  the coefficients $B_\ell$ and variables $y_\ell$ can be replaced with self-adjoint coefficients and symmetric
  variables in such a way as to obtain a  linear pencil $L$  such that $\proj_x \cD_\Delta= \proj_x \cD_L$.

\begin{lemma}\label{prop:liftMeToSym}
 Given a linear matrix polynomial $\Delta(x,y)$ as in \eqref{eq:mixedUp}  in $g$ symmetric and $h$ free variables,  there
  exists a linear pencil $L(x,w)$ in $g+2h$ variables  such that 
\[
\proj_x \cD_\Delta= \proj_x \cD_L.
\]
\end{lemma}

\begin{proof}
Write $B_\ell = C_\ell + i D_\ell$ for $C_\ell,D_\ell\in
\sac k$, and let $y_\ell =   w_{\ell} + i w_{-\ell}$, where
$w=(w_{-h},\ldots,w_{-1},$ $w_1,\ldots,w_h)$ are free symmetric variables. Then
\[
\begin{split}
B_\ell y_\ell + B_\ell^* y_\ell^* & = 
(C_\ell + i D_\ell) (  w_{\ell} + i w_{-\ell} ) + 
(C_\ell - i D_\ell) (  w_{\ell} - i w_{-\ell} ) \\
&= 2 ( C_\ell w_\ell - D_\ell w_{-\ell} ).
\end{split}
\]
Let 
\beq\label{eq:symComplex}
L(x,w) = A_0 +\sum_{j=1}^g A_j x_j 
+ 2 \sum_{\ell=1}^h \big( C_\ell w_\ell - D_\ell w_{-\ell})  
    \in \sac k \langle x,w \rangle.
\eeq
This is a linear pencil in symmetric variables with  self-adjoint coefficients.
By  construction,
\[
 \proj_x \cD_{L} = \proj_x \cD_{\Delta}. \qedhere
  \]
\end{proof}

\begin{proof}[Proof of Corollary {\rm\ref{cor:seqlasserre}}]
  There is a linear matrix polynomial $\Delta$ as in  \eqref{eq:mixedUp} such that 
  \[\Lift_p(\textvisiblespace;d)=\cD_\Delta\] and
  thus
  \[
 \hLift_p(\textvisiblespace;d)=\proj_x \cD_\Delta.\] 
Hence Lemma \ref{prop:liftMeToSym} yields a linear pencil $L_d$ with
  \[
\proj_x \cD_{L_d} = \proj_x \cD_{\Delta} =  \hLift_p(\textvisiblespace;d). \qedhere
\]
\end{proof}


\begin{remark}\label{rem:CvsR} {\it Free Sets in Real Variables.}
  The first part of  Corollary \ref{cor:seqlasserre} holds when 
  complex scalars, and thus complex self-adjoint matrices as well as complex polynomials,
  are replaced by real  scalars, symmetric matrices and real polynomials. 
  Call a linear matrix polynomial of the type in Equation \eqref{eq:mixedUp} a
  pencil in \df{mixed variables}. If the notion of a spectrahedrop is relaxed to
  include the projection of the positivity set $\cD_\Delta$ a pencil $\Delta$
  in mixed variables onto the $x$ (symmetric) variables,
  then the second part of Corollary \ref{cor:seqlasserre} holds over $\mathbb R$ too. 

  In the real setting, the construction of Lemma \ref{prop:liftMeToSym} expresses
  $B_\ell = C_\ell + D_\ell$ where $C_\ell$ is a symmetric matrix and $D_\ell$ is a skew-symmetric 
  matrix.  Thus, a mixed variable pencil can be replaced by a mixed variable pencil
  which is the sum of a  linear pencil in symmetric coefficients and variables and
  a homogeneous linear polynomial in skew-symmetric coefficients and variables. 
\end{remark}

\subsection{Examples}\label{subsec:hankelFun}

Here we explicitly write down the first Lasserre--Parrilo lift for the  \bTV.
For convenience, 
 a word $x_{i_1}x_{i_2}\cdots x_{i_k}$
will be denoted by $i_1\,i_2\,\ldots\,i_k$ and the corresponding moment by
$Y_{i_1\,i_2\,\ldots\,i_k}$. 

\subsubsection{The $d=0$ relaxation}
\label{sec:hankelFun}
 $\hLift_p(\tvs;0)$.
We first apply the lifting construction to $d=0$, and $p=1-x_1^2-x_2^4$.
Since $\deg p=4$, the lift can be written as
\begin{equation}\label{eq:las0}
\begin{split}
H_2(Y)=
\begin{pmatrix} 
1 & X_1 & X_2 & Y_{11} & Y_{12} & Y_{21} & Y_{22} \\[.1cm]
X_1 & Y_{11} & Y_{12} & Y_{111} & Y_{112} & Y_{121} & Y_{122} \\[.1cm]
X_2 & Y_{21} & Y_{22} & Y_{211} & Y_{212} & Y_{221} & Y_{222} \\[.1cm]
Y_{11} & Y_{111} & Y_{112} & Y_{1111} & Y_{1112} & Y_{1121} & Y_{1122} \\[.1cm]
Y_{21} & Y_{211} & Y_{212} & Y_{2111} & Y_{2112} & Y_{2121} & Y_{2122} \\[.1cm]
Y_{12} & Y_{121} & Y_{122} & Y_{1211} & Y_{1212} & Y_{1221} & Y_{1222} \\[.1cm]
Y_{22} & Y_{221} & Y_{222} & Y_{2211} & Y_{2212} & Y_{2221} & Y_{2222} 
\end{pmatrix}\succeq0,\\
\hs_{p,0}(Y)=1-Y_{11}-Y_{2222}\succeq0.\quad\qquad\quad\qquad
\end{split}
\end{equation}
It is well known (see e.g.~\cite{HN08})  that \eqref{eq:las0} is exact at the scalar level, meaning that 
$(X_1,X_2)\in\cD_p(1)$ if and only if  \eqref{eq:las0} has a solution.

Consider the following cut-down of \eqref{eq:las0}:

\begin{equation}\label{eq:las1}
\begin{split}
\check H_2(Y)=
\begin{pmatrix} 
1 & X_1 & X_2 &  Y_{22} \\[.1cm]
X_1 & Y_{11} & Y_{12} &  Y_{122} \\[.1cm]
X_2 & Y_{21} & Y_{22} & Y_{222} \\[.1cm]
Y_{22} & Y_{221} & Y_{222} & Y_{2222} 
\end{pmatrix}\succeq0, \\
\hs_{p,0}(Y)=1-Y_{11}-Y_{2222}\succeq0.\quad\qquad
\end{split}
\end{equation}
We shall see later that the lifts
given by 
\eqref{eq:las0} and \eqref{eq:las1} are equivalent, and are equivalent to the
standard LMI lift for the TV screen; see Section \ref{sec:ex} 
below for details.

\subsubsection{The $d=1$ relaxation} $\hLift_p(\tvs;1)$.
Here is the next Lasserre--Parrilo relaxation:
\[
\resizebox{.99\hsize}{!}{$
H_3(Y)=
\left(
\begin{smallmatrix}
 1 & X_1 & X_2 & Y_{11} & Y_{12} & Y_{21} & Y_{22} & Y_{111} & Y_{112} & Y_{121} &
   Y_{122} & Y_{211} & Y_{212} & Y_{221} & Y_{222} \\[.2cm]
 X_1 & Y_{11} & Y_{12} & Y_{111} & Y_{112} & Y_{121} & Y_{122} & Y_{1111} & Y_{1112} &
   Y_{1121} & Y_{1122} & Y_{1211} & Y_{1212} & Y_{1221} & Y_{1222} \\[.2cm]
 X_2 & Y_{21} & Y_{22} & Y_{211} & Y_{212} & Y_{221} & Y_{222} & Y_{2111} & Y_{2112} &
   Y_{2121} & Y_{2122} & Y_{2211} & Y_{2212} & Y_{2221} & Y_{2222} \\[.2cm]
 Y_{11} & Y_{111} & Y_{112} & Y_{1111} & Y_{1112} & Y_{1121} & Y_{1122} & Y_{11111} &
   Y_{11112} & Y_{11121} & Y_{11122} & Y_{11211} & Y_{11212} & Y_{11221} & Y_{11222} \\[.2cm]
 Y_{21} & Y_{211} & Y_{212} & Y_{2111} & Y_{2112} & Y_{2121} & Y_{2122} & Y_{21111} &
   Y_{21112} & Y_{21121} & Y_{21122} & Y_{21211} & Y_{21212} & Y_{21221} & Y_{21222} \\[.2cm]
 Y_{12} & Y_{121} & Y_{122} & Y_{1211} & Y_{1212} & Y_{1221} & Y_{1222} & Y_{12111} &
   Y_{12112} & Y_{12121} & Y_{12122} & Y_{12211} & Y_{12212} & Y_{12221} & Y_{12222} \\[.2cm]
 Y_{22} & Y_{221} & Y_{222} & Y_{2211} & Y_{2212} & Y_{2221} & Y_{2222} & Y_{22111} &
   Y_{22112} & Y_{22121} & Y_{22122} & Y_{22211} & Y_{22212} & Y_{22221} & Y_{22222} \\[.2cm]
 Y_{111} & Y_{1111} & Y_{1112} & Y_{11111} & Y_{11112} & Y_{11121} & Y_{11122} &
   Y_{111111} & Y_{111112} & Y_{111121} & Y_{111122} & Y_{111211} & Y_{111212} &
   Y_{111221} & Y_{111222} \\[.2cm]
 Y_{211} & Y_{2111} & Y_{2112} & Y_{21111} & Y_{21112} & Y_{21121} & Y_{21122} &
   Y_{211111} & Y_{211112} & Y_{211121} & Y_{211122} & Y_{211211} & Y_{211212} &
   Y_{211221} & Y_{211222} \\[.2cm]
 Y_{121} & Y_{1211} & Y_{1212} & Y_{12111} & Y_{12112} & Y_{12121} & Y_{12122} &
   Y_{121111} & Y_{121112} & Y_{121121} & Y_{121122} & Y_{121211} & Y_{121212} &
   Y_{121221} & Y_{121222} \\[.2cm]
 Y_{221} & Y_{2211} & Y_{2212} & Y_{22111} & Y_{22112} & Y_{22121} & Y_{22122} &
   Y_{221111} & Y_{221112} & Y_{221121} & Y_{221122} & Y_{221211} & Y_{221212} &
   Y_{221221} & Y_{221222} \\[.2cm]
 Y_{112} & Y_{1121} & Y_{1122} & Y_{11211} & Y_{11212} & Y_{11221} & Y_{11222} &
   Y_{112111} & Y_{112112} & Y_{112121} & Y_{112122} & Y_{112211} & Y_{112212} &
   Y_{112221} & Y_{112222} \\[.2cm]
 Y_{212} & Y_{2121} & Y_{2122} & Y_{21211} & Y_{21212} & Y_{21221} & Y_{21222} &
   Y_{212111} & Y_{212112} & Y_{212121} & Y_{212122} & Y_{212211} & Y_{212212} &
   Y_{212221} & Y_{212222} \\[.2cm]
 Y_{122} & Y_{1221} & Y_{1222} & Y_{12211} & Y_{12212} & Y_{12221} & Y_{12222} &
   Y_{122111} & Y_{122112} & Y_{122121} & Y_{122122} & Y_{122211} & Y_{122212} &
   Y_{122221} & Y_{122222} \\[.2cm]
 Y_{222} & Y_{2221} & Y_{2222} & Y_{22211} & Y_{22212} & Y_{22221} & Y_{22222} &
   Y_{222111} & Y_{222112} & Y_{222121} & Y_{222122} & Y_{222211} & Y_{222212} &
   Y_{222221} & Y_{222222}
\end{smallmatrix}
\right) \succeq0,$
}
\]
$$
\hs_{p,1}(Y)=
\left(
\begin{smallmatrix}
1 -Y_{11}-Y_{2222} & X_1-Y_{111}-Y_{22221} & X_2-Y_{112}-Y_{22222} \\[.2cm]
 X_1-Y_{111}-Y_{12222} & Y_{11}-Y_{1111}-Y_{122221} &
Y_{12}-Y_{1112}-Y_{122222} \\[.2cm]
 X_2-Y_{211}-Y_{22222} & Y_{21}-Y_{2111}-Y_{222221} &
Y_{22}-Y_{2112}-Y_{222222}
\end{smallmatrix}
\right)\succeq0.
$$

\subsection{Proof of Theorem \ref{thm:intersect} }
  The following lemma generalizes Lemma \ref{lem:archtobound}.

\begin{lem}
 \label{lem:archdegbds}\rm
   If $p$ is archimedean, then there is a natural number
   $\nu$ and a positive number $C$ such that if
\[
  Y\in\Lift_p(n;d)
\]
 and $\alpha$ is a word with length $|\alpha|$
  at most $2(d-\nu)$, 
  then 
\[
  \|Y_\alpha\| \le C^{|\alpha|}.
\]
\end{lem}

  The proof of this lemma
  uses the following variant of the Gelfand-Naimark-Segal (GNS) construction. A 
  Hankel matrix 
$  Y\in\Lift_p(n;d)$ with a truncated positive semidefiniteness property
  generates a pre-Hilbert space as follows.
  Assuming that $H_{d + \left\lceil\frac12\deg p\right\rceil}(Y)$ and $ \hs_{p,d}(Y)$ are  positive semidefinite,  define the sesquilinear 
  form on $\FF^n \otimes \FF\langle x\rangle_{d+\dpovertwo}$ by
\[
  \langle h\otimes \alpha, k\otimes \beta \rangle = \langle Y_{\beta^*\alpha}h,k\rangle.
\]
  Positivity of $\hs_{p,d}(Y)$ is then equivalent to the condition,
\[
   \left\langle \sum h_\alpha \otimes p\alpha, \sum h_\beta \otimes \beta \right\rangle \ge 0
\]
 for all $h\in\FF^n \otimes \FF\langle x\rangle_{d+\dpovertwo}$ of the form
\[
  h = \sum_{|\alpha|\le d} h_\alpha \otimes \alpha. 
\]
 In particular, if $f$ is a polynomial of degree $\nu$ and
\[
  q = f^* p f,
\]
 then for $h=\sum_{|\alpha|\le d-\nu} h_\alpha \otimes \alpha$,
\[
  \big\langle h_\alpha \otimes f^* p f \alpha,\sum_\beta h_\beta \otimes \beta\big\rangle
   = \big\langle h_\alpha \otimes pf \alpha, \sum_\beta h_\beta \otimes f\beta \big\rangle \ge 0.
\]
 Hence the localizing matrix $\hs_{q,d-\nu}(Y)$ is positive semidefinite. 
 An analogous statement is true for a polynomial $q=s^*s$ when
 $s$ has degree at most $\nu+\dpovertwo$. 

\begin{proof}[Proof of Lemma {\rm{\ref{lem:archdegbds}}}]
  By the archimedean hypothesis, there exist a constant $C$
  and a natural number $\mu$ such that for each $j$ 
 there exist polynomials $s_{j,1},\dots,s_{j,\mu}$
   and $f_{j,1},\dots,f_{j,\mu}$ with
\begin{equation}\label{eq:archCert}
 q_j= C^2-x_j^2 = \sum_k s_{j,k}^* s_{j,k} + \sum_{\ell} f_{j,\ell}^* p f_{j,\ell}.
\end{equation}
  Choose $\nu$ such that $\deg(f_{j,\ell}) \le \nu$ for all $j,\ell$ and
  $\deg(s_{j,k})\le \nu +\dpovertwo$ for all $j,k$. 
  Fixing $j$ and letting $S_k = s_{j,k}^*s_{j,k}$ and $F_k=f_{j,\ell}^* pf_{j,\ell}$ 
  it follows that 
\[
  \hs_{q_j,d-\nu}(Y) = \sum \hs_{S_k,d-\nu}(Y) + \sum \hs_{F_\ell,d-\nu}(Y).
\]
  Hence, by the discussion above, the localizing matrix $\hs_{q_j,d-\nu}(Y)$ is
  positive semidefinite.

  Given a $\beta$ with $|\beta|< d-\nu$, positivity
  of the localizing matrix for $C^2-x_j^2$ implies that
\[
  C^2 Y_{\beta^* \beta} \succeq Y_{x_j^*\beta^*\beta x_j}.
\]
  Thus an induction argument on $|\beta|$ gives, for any $|\beta|\le d-\nu$
  that
\[
   Y_{\beta^* \beta} \preceq C^{2|\beta|} I.
\]
  Now suppose $\alpha$ is a word with $|\alpha|\le 2(d-\nu)$.
  There exist words $\beta$ and $\gamma$ of length
  at most $d-\nu$  such that $\alpha = \beta^*\gamma$. From
  the fact that  $H(Y)\succeq0$, it follows that
\[
 \begin{pmatrix}  Y_{\beta^* \beta} & Y_{\beta^* \gamma} \\ Y_{\gamma^* \beta} & Y_{\gamma^* \gamma} \end{pmatrix}
\]
 is positive semidefinite. Thus,
\[
  Y_{\gamma^*\beta} Y_{\beta^*\gamma} \preceq C^{2\big(|\beta|+ |\gamma|\big)} I.
\]
  The desired inequality follows. 
\end{proof}

\begin{proof}[Proof of Theorem \rm{\ref{thm:intersect}}]
It is obvious that 
\[
 \bigcap_d \Lift_p(n;d) = \Lift_p(n)
\]
in the sense that a moment sequence $(Y_\al)_\al$ all of whose truncations
satisfy the positive semidefiniteness of the Hankel matrices 
$H_{d + \left\lceil\frac12\deg p\right\rceil} (Y )$, and
$ \hs_{p,d}(Y(n))  $, is in $\Lift_p(n)$, i.e., makes the infinite Hankel matrices
$H(Y)$ and $\hs_p(Y)$ positive semidefinite.

  Suppose the moment sequence $Z\in \bigcap_d\hLift_p(n;d)$.
  In this case, for each $d$ there is a (truncated) moment sequence
  $Y^{(d)} = (Y^{(d)}_\alpha)\in \hLift_p(n;d)$ such that
 \[
   (Y^{(d)}_{x_1},\dots,Y^{(d)}_{x_g}) = Z.
\]
  By construction, for each $\alpha$ 
  the sequence $(Y^{(d)}_\alpha)_{2d\ge |\alpha|+\deg(p)}$ is bounded. 
  Since we have countably many such 
  sequences, there is a subsequence $(d_k)$ with the property that
  $(Y^{(d_k)}_\alpha)$ converges termwise (in $\alpha$ with $k$ tending to $\infty$).
  This limit moment sequence $Y$ will be in $\Lift_p(n)$ and moreover,
\[
 Z = (Y_{x_1},\dots,Y_{x_g}) = \hat{Y}
\]
 so that $Z\in \hLift_p(n)$. 
\end{proof}

\subsection{Truncated Quadratic Modules and the  BPCP}

 Given $\alpha,\beta,\mu\in\N$, and an $\ell\times\ell$ free matrix polynomial $p$, set 
\begin{equation}
 \label{eq:Malbeta}
  M_{\al, \beta}^{\mu}(p):=
   \Sigma_{\al}^{\mu}  +\Big\{ \sum_i^{\rm finite}  f_i^* pf_i :
   \  f_i \in \FF^{\ell\times \mu}\ax_\beta \Big\}
  \  \subseteq \ \FF^{\mu\times \mu}\ax_{\max \{2\al, 2\beta +a\}},
\end{equation}
 where $a=\deg(p)$ and $\Sigma_\al^\mu$ denotes all $\mu\times\mu$ sums
of squares of degree $\leq2\al$.
Obviously, if $f\in M_{\al,\beta}^{\mu}(p)$ then $f|_{\cD_p}\succeq0$.
We call $M_{\al,\beta}^\mu(p)$ the \df{truncated quadratic module} defined by $p$.
For notational convenience, we write $M_{k}$ for  $M_{\al,\be}$ with 
$k=\max\{2\al,2\be+a\}$.
We also introduce
\[
 M^\mu(p):=\bigcup_{\al,\be} M_{\al,\beta}^\mu(p),
\]
the \df{quadratic module} defined by $p$. 
If $\mu=1$ we shall often omit the superscript $\mu$.
Observe that $p$ is archimedean if 
the convex cone $M^\mu(p)$ has an order unit, i.e., for all symmetric $\mu\times\mu$ matrix
polynomials $f$ there is $N\in\N$ with $N-f\in M^\mu(p)$. (This notion is 
easily seen to be 
independent of $\mu$, cf.~\cite[\S 6]{HKM+}.)

\begin{defn}
Let $\mu,N\in\N$. We say that $p$ has the \df{$(N,\mu)$-bound positivity certificate property  (BPCP)},
if for every $\mu\times\mu$ linear pencil $L$, we have
\[
L|_{\cD_p}\succeq0 \quad\iff\quad L \in M_{N}^\mu(p).
\]
If $N$ can be chosen  independently of $\mu$, then we say $p$ has the 
\df{$N$-BPCP}.
\end{defn}

We refer the reader to \cite{Scw04,Nis07} for the classical commutative study
of degree bounds needed in Positivstellensatz certificates.

\subsubsection{A sufficient stopping criterion for the free Lasserre--Parrilo lift}

\begin{lem}
 \label{lem:hLiftclosed}\rm
  If, for a positive integer $n$, 
  the set $\Lift_p(n)$ is bounded in the sense that 
  for each $\alpha$ there exists a $C_\alpha$ such that
   $\|Y_\alpha \| \le C_\alpha$ for all $Y\in\Lift_p(n)$,
   then $\hLift_p(n)$ is compact. 
\end{lem}

\begin{proof}
  With the boundedness hypothesis, the set $\Lift_p(n)$, viewed
  as a subset of the product space $\prod_{\alpha} \smatn$ 
  is entrywise bounded.  It is also seen to be entrywise 
  closed. Thus it is a product of compact sets and therefore
  compact.  Consequently the projection $Y\mapsto \hY$
  being the finite product of the projections determined by
  the $x_j$ has compact range; i.e., $\hLift_p(n)$ is compact.
\end{proof}

The next theorem says if  $N$-BPCP holds, then one of the truncated
Lasserre--Parrilo lifts gives exactly the free convex hull of $\cD_p$.

\begin{thm} \label{thm:BPCPliftStops}
If $\cD_p$ is uniformly bounded, and $p$ has the $N$-BPCP, 
then \[\cmco \cD_p=\hat \Lift_p\Big(\tvs; \left\lceil \frac N2\right\rceil\Big).\]
\end{thm}

\begin{proof}
Let $\eta= \lceil \frac N2\rceil$. 
Clearly, $\cD_p\subseteq\hLift_p(\tvs;\eta)$, and since $\hLift_p(\tvs;\eta)$ is matrix convex, 
$\mco \cD_p\subseteq\hLift_p(\tvs;\eta)$. 
Since $\cD_p$ is uniformly bounded and $p$ has the $N$-BPCP,
$p$ is archimedean. Hence by Lemma \ref{lem:archdegbds},
$\Lift_p(\tvs;\eta)$ is compact (e.g.~in the product topology), and hence
$\cmco \cD_p\subseteq\ol{\hLift_p(\tvs;\eta)}=\hLift_p(\tvs;\eta)$.

Now assume $Y\in\hLift_p(\tvs;\eta)\setminus\cmco \cD_p$,
and choose $W\in\Lift_p(\tvs;\eta)$ satisfying $\hat W=Y$.
 Suppose $Y$ is a $g$-tuple of size
$\mu\times\mu$ matrices.
By the Hahn-Banach Theorem \ref{prop:sharp} there 
is a linear pencil $L$ (of size  $\mu$) with $L|_{\mco \cD_p}\succeq0$ and $L(Y)\not\succeq0$. 
By the $N$-BPCP property for $p$, 
we have that $L \in M_{N}^\mu(p)$, i.e.,  
\begin{equation}\label{eq:symbol}
L= \sum_k h_k^*h_k + \sum_{i=1}^r  f_i^* pf_i .
\end{equation}
Here $\deg(h_k)\leq \lfloor \frac N 2\rfloor$  and $2\deg (f_i)+\deg(p)\leq N$ for $i=1,\ldots,r$.
Now apply the Riesz map $\Phi_W^\mu$ to \eqref{eq:symbol}:
\beq\label{eq:las2}
\Phi_W^\mu(L) = \sum_k \Phi_W^\mu(h_k^*h_k) + \sum_{i=1}^r \Phi_W^\mu( f_i^* pf_i ).
\eeq
Since $H_\eta(W)\succeq0$ and $\hs_{p,\eta}(W)\succeq0$, Proposition \ref{prop:riesz}
implies the right hand side of \eqref{eq:las2} is positive semidefinite.
On the other hand, since $L$ is linear, 
$\Phi_W^\mu(L)= L(\hat W)=L(Y)\not\succeq0$, a contradiction.
\end{proof}

\subsubsection{More on the Positivstellensatz}\label{subsec:LPP}

 The polynomial $p$ has the \df{linear Positivstellensatz property}
 (\df{LPP}) 
  if whenever $L$ is a monic linear pencil 
  positive semidefinite
on $\cD_p$, then for each $\epsilon>0$ there
  exists natural numbers $n_s$ and $n_f$ and matrix polynomials
  $s_1,\dots,s_{n_s}$ and $f_1,\dots,f_{n_f}$ such that
\[
  L+\epsilon =\sum_{j=1}^{n_s} s_j^* s_j  + \sum_{j=1}^{n_f} f_j^* p f_j.
\]
(So $L+\epsilon\in M^\mu(p)$.) 
  Note that the LPP condition is weaker than the BPCP.

\begin{prop}
 \label{prop:ifposstatz}
   Suppose 
$\cD_p$ is uniformly bounded and $p$ has the LPP. Then $\cmco \cD_p = \hLift_p$.
\end{prop}

\begin{proof}
Observe that the uniform boundedness of $\cD_p$ together with the LPP implies
$p$ is archimedean.
 Suppose $L$ is positive semidefinite on $\cD_p$. By the
  LPP, 
\begin{equation}
 \label{eq:LPPrep}
 \epsilon + L  = \sum s_j^* s_j  + \sum f_j^* p f_j.
\end{equation}
 On the other hand, if $X\in \hLift_p$, then by 
   Lemma \ref{lem:basiclasserre}
    there exists a $Z\in\cO_p$ and an isometry
  $V$ such that $X=V^*ZV$.   Because
  $Z\in\cO_p$, it follows from the representation \eqref{eq:LPPrep}, that  $\epsilon +L(Z) \succeq 0$. 
 Since $\eps>0$ was arbitrary, this shows $L(Z)\succeq 0$. 
  Hence by Corollary \ref{cor:hb}, $X\in \cmco \cD_p$.
\end{proof}

\section{Examples}\label{sec:ex}

In this section we present a few examples, starting with a detailed study of the TV screen and
its ``classical'' spectrahedral lifts, see Subsections \ref{subsec:71} and \ref{subsec:72}. We show that, unlike in the commutative settings, 
the first Lasserre--Parrilo lift is not exact. Then in Subsection \ref{subsec:73} we prove
that the matrix convex hull of the TV screen is dense in its operator convex hull.
Finally, Subsection \ref{subsec:74} contains  simple examples where the Lasserre--Parrilo
lifts are exact.

\subsection{The  Bent TV Screen}\label{subsec:71}

Recall the \bTV, 
\[
p=1-x^2-y^4.
\]
The corresponding free semialgebraic set $\cD_p$ is called the
 \df{TV screen}.

\begin{lem}
\label{lem:5/4}
$p$ is $\frac54$-archimedean.
\end{lem}

\begin{proof}
Simply note that
\[
\frac54-x^2-y^2= \left(y^2-\frac12\right)^2+ (1-x^2-y^4).
\qedhere
\]
\end{proof}

The usual lift of $\cD_p(1)=\conv\big(\cD_p(1)\big)$ is given by $\cD_\La(1)$, where
\[
\La= \begin{pmatrix} 
 1 & 0 &   x \\
 0 & 1 & w \\
x & w & 1
 \end{pmatrix}
 \oplus
 \begin{pmatrix}
 1 &    y \\
 y  & w
 \end{pmatrix}.
\]
However, $\La$ is not monic, so we modify the construction somewhat.
Let
\[
L_1(x,y,w)=\begin{pmatrix}
 1 &  \gamma  y \\
 \gamma y  & w+\alpha
 \end{pmatrix}
 , \quad
 L_2(x,y,w)=
 \begin{pmatrix}
 1 & 0 &  \gamma ^2 x \\
 0 & 1 & w \\
  \gamma ^2 x & w & 1-2  \alpha w
 \end{pmatrix}
 \]
 where $\al>0$ and $1+\al^2=\ga^4$, and set $L=L_1\oplus L_2$. 
 While strictly speaking $L$ is not monic, the free spectrahedron $\cD_L$ contains $0$ in its interior, so $L$ can be
easily modified to become monic.
 It is  worth noting that 
\begin{equation}\label{eq:standardTVlift}
\La(X,Y,W)\succeq0 \qquad\iff\qquad W\succeq Y^2 \;\text{   and   }\; 1-X^2-W^2\succeq0
\end{equation}
as is easily seen by using Schur complements.
 
Let $\cC$ denote the free spectrahedrop obtained as the projection of $\cD_L$ onto the first two coordinates. Thus,
\begin{equation}\label{eq:spectrahedrop}
\cC=\{(X,Y)\in\mbS^2:\exists W\in\mbS \text{ such that } L(X,Y,W)\succeq0\}.
\end{equation}
It is  easy  to see
$
\cC=\{(X,Y)\in\mbS^2:\exists W\in\mbS \text{ such that }  \La(X,Y,W)\succeq0\}.$

The main result of this section is:

\begin{thm}\label{thm:TV}
$\cmco\cD_p=\mcop\subsetneq\cC$.
\end{thm}

We shall prove the equality in Subsection \ref{subsec:73} below, and now proceed 
to establish the strict
inclusion.

\begin{lem}\label{lem:preTV}
\mbox\par{}
\ben[\rm(1)]
\item
The projection $\cC(1)$ of $\cD_L(1)$ onto the $(x,y)$-space equals $\cD_p(1)$.
\item
$\mco \cD_p\subseteq\cC$.
\een
\end{lem}

 \begin{proof}
 Given $(x,y)\in\cD_p(1)$, let $w=\gamma^2y^2-\al$. This makes $L_1(x,y,w)$ positive semidefinite and singular. The Schur complement of the top $2\times2$ block of $L_2(x,y,w)$ is thus
 \[
1-2\al w- \ga^4 x^2-w^2=
1+ \alpha ^2-\gamma ^4 x^2- \gamma ^4 y^4=\ga^4(1-x^2-y^4)\geq0
 \]
 making $L_2(x,y,w)\succeq0$.
 
 Conversely, if $(x,y,w)\in\cD_L(1)$, then $w\geq\gamma^2y^2-\al$. Again, by way of Schur complements, 
 \[
 \begin{split}
 0&\leq 1-2\al w- \ga^4 x^2-w^2
 = 1+\al^2 - (\al+w)^2 - \ga^4 x \\
& \leq 1+\al^2 - \ga^4 y^4 -\gamma ^4 x^2
 = \ga^4(1-x^2-y^4),
 \end{split}
 \] 
 showing $1-x^2-y^4\geq 0$.
 
 For (2), take $(X,Y)\in\cD_p$.
Thus $I-X^2-Y^4\succeq0$. Set $W=\gamma^2Y^2-\al I$.
This makes $L_1(X,Y,W)\succeq0$. 
The Schur complement of the block top $2\times2$ block of $L_2(X,Y,W)$ is thus
 \[
1-2\al W- \ga^4 X^2-W^2=
1+ \alpha ^2-\gamma ^4 X^2- \gamma ^4 Y^4=\ga^4(1-X^2-Y^4)\succeq0
 \]
 making $L_2(X,Y,W)\succeq0$.
 Since $\cC$ is matrix convex, this establishes $\mco \cD_p\subseteq\cC$.
 \end{proof}

\begin{lem}\label{lemProp:TV}
$\mcop\subsetneq\cC$.
\end{lem}

\begin{proof}
For this strict inclusion we simply exhibit  matrix tuples, namely,
 points in the projection $\cC$ onto the $(x,y)$-space of $\cD_L$ which are not
 in $\mcop$.
 In terms of $\mu>0$ specified below, let 
\[
   Y= \sqrt{\mu}  \begin{pmatrix} 1 & 0 \\ 0 & 0 \end{pmatrix}.
\]
 Take
\[
  W = \mu \begin{pmatrix} 2 & 1 \\ 1 & 1 \end{pmatrix}.
\]
  Choose $\mu$ so that the norm of $W$ is $1$ 
 and let \[X^2 = 1 - W^2.\]
   Then 
$      1-X^2 -W^2 = 0
$   and at the same time $Y^2\le W$.   
   Thus $(X,Y) \in \cC$.
   On the other hand, 
\[
  Y^4 - W^2 = \mu^2 \begin{pmatrix} 4 & 3\\ 3 & 2 \end{pmatrix} \not\succeq 0.
\]
 Hence $I-X^2-Y^4 \not\succeq 0$, i.e., $(X,Y)\not\in\cD_p$.

\def\tY{\tilde Y}
  We next show that $(X,Y)\not\in\mcop$.  It suffices to
  show if $\tX, \tY$ are of the form
\[
  \tX =\begin{pmatrix} X& \al \\ \al^*& *  \end{pmatrix}, \ \ 
   \tY=\begin{pmatrix} Y & \beta \\ \beta^* & \nu \end{pmatrix}.
\]
  then  $I-\tX^2 -\tY^4 \not\succeq 0.$
  We argue by contradiction and accordingly assume
  $I-\tX^2-\tY^4\succeq  0$. 
  To do this the first step will be  to show that $\beta =0$. 
  Next, $\beta=0$ implies,
  projecting onto the top subspace,
\[
   0 \preceq  I - (X^2 + \al\al^*)  - Y^4 \preceq  
      I - X^2   - Y^4.
\]
  But then, because
 $I-X^2 -Y^4 \not\succeq 0$, we get a contradiction.

Now to the attack on $\be$.
  Note that 
\begin{equation}
\label{eq:YY}
  \tY^2 = \begin{pmatrix}  Y^2 + \beta \beta^* &\delta
  \\  \delta^* &* \end{pmatrix}.
\end{equation} 
for some $\delta$ and some $*$.
 Let $T : =Y^2+\beta \beta^* \succeq Y^2$.  Further, note that 
\[
  \tY^4 = \begin{pmatrix} T^2 +\delta \delta^* & * \\ * & * \end{pmatrix}.
\]
 The upper left entry of  $I-\tX^2 -\tY^4$
 equals 
\begin{equation}
 \label{eq:XT}
 0 \preceq
  I - (X^2 + \al\al^* ) - (T^2 +\delta \delta^*) \preceq
     I- X^2 - T^2  = W^2-T^2.
\end{equation}
 Further, we have
\[
  Y^2 =\mu \begin{pmatrix} 1 & 0 \\ 0 & 0 \end{pmatrix}.
\] 
  So after dividing \eqref{eq:XT} through by $\mu^2$, we obtain,
\begin{equation}\label{eq:weirdoIneq}
  \begin{pmatrix} 2 & 1 \\ 1 & 1 \end{pmatrix}^2
   \succeq \left( \begin{pmatrix} 1 & 0 \\ 0 & 0 \end{pmatrix}
    +  \frac{1}{\mu} \beta \beta^* \right)^2.
\end{equation}
Since the square root function is operator monotone, \eqref{eq:weirdoIneq}
yields
\begin{equation}\label{eq:lessWeird}
\bem 1 \\ 1 \eem  \bem 1 & 1 \eem =
  \begin{pmatrix} 1 & 1 \\ 1 & 1 \end{pmatrix}
   \succeq   \frac{1}{\mu} \beta \beta^*,
   \end{equation}
 or equivalently,
\[
 \beta = \sqrt{\mu}  \begin{pmatrix} 1 \\ 1 \end{pmatrix} b^*,
\]
 for some  vector $b$ with norm $\leq1$.
 Putting these back into \eqref{eq:weirdoIneq} leads to
 \[
 \bem
 -2 \| b\| ^4-2 \| b\| ^2+4 & -2 \| b\| ^4-\| b\| ^2+3 \\
 -2 \| b\| ^4-\| b\| ^2+3 & 2-2 \| b\| ^4
 \eem\succeq0.
 \]
 Since the determinant of this matrix equals
 \[
 -(\| b\|^2 -1)^2, 
 \]
 we see $\|b\|=1$.  In particular, we have equality in \eqref{eq:lessWeird} and \eqref{eq:XT}.
 Hence $T=W$ so that $I-X^2-T^2 =0$. 
 
 Returning  to 
 the upper left hand entry of  $I-\tX^2 -\tY^4$, it follows
 from \eqref{eq:YY}  and \eqref{eq:XT} that
 we have
\[
 I-X^2 - T^2 -\delta \delta^* \succeq 0.
\]
  Hence $\delta \delta^*=0$ and so $\delta=0$.
  Since $\delta$ is of the form 
  $\delta = Y \beta +\beta \nu,$ we have
\[
 \begin{split}
   0 &=  Y\beta + \beta \nu       =  \sqrt{\mu} \begin{pmatrix} 1 & 0 \\ 0 & 0 \end{pmatrix} \beta
         +\beta\nu \\
 & = \mu  \begin{pmatrix} 1 & 0 \\ 0 & 0 \end{pmatrix}  
 \begin{pmatrix} 1 \\ 1 \end{pmatrix} b^* + \sqrt{\mu}  \begin{pmatrix} 1 \\ 1 \end{pmatrix} b^* \nu \\
  & = \mu \bem 1\\0\eem b^* + \sqrt{\mu} \begin{pmatrix} 1 \\ 1 \end{pmatrix} b^* \nu,
 \end{split}
\]
leading to 
\[
b^*\nu = 0 \quad \text{and} \quad b^*\nu+\sqrt\mu b^*=0.
\]
Hence $b^*=0$. This implies $\be=0$, delivering the promised contradiction. 
\end{proof}

\begin{prop}
$\mco \cD_p(1)=\cD_p(1)$.
\end{prop}

\begin{proof}
This follows from 
Lemma \ref{lem:preTV} and Proposition \ref{prop:liftMe}.
Alternately, use
$\cD_p\subseteq\mco \cD_p\subseteq\cC$ together with item (1) of Lemma \ref{lem:preTV}.
\end{proof}

\subsection{Comparing the $L$-Lift with the Lasserre--Parrilo Relaxations}\label{subsec:72}
Two Lasserre-Parrilo lifts of the \bTV{} were proposed in Subsection \ref{subsec:hankelFun}.
The malicious point 
constructed in the proof of Lemma \ref{lemProp:TV}
 serves to show that the  lift $\Lift_p(\tvs;0)$ based on 
 $\hs_{p,0}(Y)\oplus H_2(Y)\succeq0$ is again inexact, i.e., its projection 
 $\hLift_p(\tvs;0)$ 
 is still strictly bigger than $\mco \cD_p$. 
 On the other hand, the second Lasserre--Parrilo relaxation
 $\hs_{p,1}(Y)\oplus H_3(Y)\succeq0$ does seem to separate the malicious point from
 $\mco \cD_p$ -- according to our computer experiments.

\begin{prop}
Let $p=1-x_1^2-x_2^4$. Then $\hLift_p(\textvisiblespace;0)=\cC$, while
$\hLift_p(2;1)\subsetneq \cC(2)$.
\end{prop}

\begin{proof}
Let $\Lift_p'(\tvs;0)$ denote the ``reduced'' lift obtained by using \eqref{eq:las1}, and
$\hLift_p'(\tvs;0)$ its projection. It is clear that
$\hLift_p'(\tvs;0)\supseteq \hLift_p(\tvs;0)$. 
Next, assume  $(X_1,X_2)\in \hLift_p'(\textvisiblespace;0)$, and take a feasible point
$Y$ for \eqref{eq:las1}. Then with $W=Y_{22}$ we have
$W\succeq X_2^2$ by considering the submatrix of $\check H_2(Y)$ spanned by columns and rows $1,3$. 
Likewise, $Y_{11}\succeq X_1^2$ and $Y_{2222} \succeq W^2$. Hence
\[
0\preceq 1-Y_{11}-Y_{2222}\preceq 1-X_1^2-W^2,
\]
showing $\La(X_1,X_2, W)\succeq0$, i.e., $(X_1,X_2)\in \cC$.

Conversely, let $(X_1,X_2)\in\cC$. Choose $Y$ so that
\[
\check H_2(Y) =\bem
1 & X_1 & X_2 & W \\
X_1 & X_1^2 & X_1 X_2 & X_1 W\\
X_2 & X_2 X_1 & W & X_2 W \\
W & W X_1 & W X_2 & W^2
\eem.
\]
Then \[
\hs_{p,0}(Y)=1-X_1^2-W^2\succeq0
\]
by assumption. Furthermore,
\[
\check H_2(Y)= \bem 1 & 0 & X_2 & W \\ 0 & 1 & 0 & 0 \eem^*
\bem 1 & X_1 \\ X_1 & X_1^2 \eem 
\bem 1 & 0 & X_2 & W \\ 0 & 1 & 0 & 0 \eem
+ \bem 0 & 0 & 0 & 0 \\ 0 & 0 & 0 & 0\\ 0 & 0 & W-X_2^2 & 0\\ 0 & 0 & 0 & 0\eem
\succeq0.
\]
All this shows $(X_1,X_2)\in\hLift_p'(\textvisiblespace;0)$.

As a final step, we extend $\check H_2(Y)$ to a positive semidefinite $H_2(Y)$. 
Again, this is now straightforward. Using
\[
Z=
\bem
X_1^2 & X_1X_2 & 0 \\
 0 & 0 & 0 \\
 0 & 0 & X_1 \\
 0 & 0 & 0
 \eem
 \]
 we set
\[
P H_2(Y) P =   \bem I_4 & Z \eem^* \check H_2(Y) \bem I_4 & Z \eem,
\]
where $P$ is the permutation matrix 
of the permutation $\bem 4 & 5 & 6 & 7\eem$.
Hence $(X_1,X_2)\in\hLift_p(\textvisiblespace;0)$, concluding the first part of the proof.

The second statement of the proposition follows from
numerical computer experiments;
see the Mathematica notebook {\tt TVlift.nb} available from {\tt arxiv}. 
\end{proof}

\subsection{Matrix versus Operator Convex Hull: Bent TV Screen}\label{subsec:73}

 From  Theorem \ref{thm:infinitelasserre}, the 
  closure of $\mco\cD_p$ is contained in $\mcop$. 
 While  this inclusion 
  is generally proper
  (e.g.~there are examples of archimedean $p$ with
  $\cD_p=\varnothing\neq\cD_p^\infty$), the proposition below says
  that these sets are the same in at least one
 non-trivial example. 
 The proof uses spectral theory for bounded self-adjoint operators on a Hilbert space.

\begin{prop} Let $p= 1 -x_1^2 -x_2^4$. Then
$\mcop(n)$ is  the closure of $\mco \cD_p(n)$.
\end{prop}

\begin{proof}
 Fix a point $X\in\mathbb S_m^2 $ in the operator convex hull of the bent TV screen. Thus,
 there a Hilbert space $\cH$ and a tuple $Y=(Y_1,Y_2)$ of bounded self-adjoint operators on $\cH$ such that
\[
  I\succeq Y_1^2 + Y_2^4,
\]
 and an isometry $V:\FF^m\to \cH$ such that $X=V^*YV$. 

  Since $Y_2$ is self-adjoint, it has a spectral decomposition,
\[
  Y_2 = \int_{-1}^1 t \, dE(t),
\]
 for  a spectral measure $E$ on the interval $[-1,1]$.
  Given a positive integer $N$, let
\[
  \omega_j^N = \left[\frac{j}{N},\frac{j+1}{N}\right)
\]
 for $-N\le j < N-1$ and let $\omega_{N-1}^N = [\frac{N-1}{N},1]$. 
 For $0\le  j$, let $t_j=\frac{j}{N}$ and for $j<0$, let $t_j=\frac{j+1}{N}$.
 Let 
\[
 Z=  \sum_{j=-N}^{N-1}  t_j E(\omega_j^N)
\]
 and observe that $Z$ and $Y_2$ commute.
 In particular,
\beq
 \label{eq:ineqCommute}
 Z^4 \preceq Y_2^4.
\eeq

 Consider finite dimensional subspaces
\[
     E(\omega_j^N) \cH \supset \cH_j =  E(\omega_j^N)V\FF^m.
\]
 Let $\cK=\bigoplus_{j=-N}^{N-1} \cH_j$. Thus,
 $\cK$ is finite dimensional and contained in $\cH$.  Further, letting
 $W:\cK \to \cH$ denote the inclusion of $\cK$ into $\cH$, 
\[
 \tilde{Y}_2 = W^* Z W
\]
 satisfies,
\[
   \tilde{Y}_2^4 = W^* Z^4 W \preceq W^* Y_2^4 W,
\]
because of \eqref{eq:ineqCommute}.
 Let $\tilde{Y}_1 = W^* Y_1 W$. It follows that
\[
   \tilde{Y}_1^2 + \tilde{Y}_2^4  \preceq Y_1^2 + Y_2^4 \preceq I.
\]

At the same time, by construction, $V$ maps into $\cK$ so
  that $W^*V$ is an isometry  and 
\[
  X_1 = (W^*V)^* (W^* Y_1W)W^*V
\]
 Thus, the pair $(W^*V)^* \tilde{Y} W^*V = (X_1,(W^*V)^* \tilde{Y_2}W^*V)$ is in 
 the bent TV screen.  

 Emphasizing the dependence of $W$ on $N$, write $W_N=W$ and $Z^N=Z$. With this notation, observe that
\[
 \| Z^N-Y_2\| = \left\| \sum t_j E(\omega^N_j) -\int t\, dE(t) \right\| \le \frac{1}{N}.
\]
 Hence, $Z^N$ converges in the strong operator topology to $Y_2$.  
 Since $W_N W_N^* V=V$, it follows that
\[
  (W_N^*V)^* \tilde{Y}_2^N W_N^* V 
  = (W_N^*V)^* W_N^* Z^N W_N (W_N^*V) = V^* Z^N V
\]
 converges to $V^*Y_2 V= X_2$.
 The conclusion is that $X$ is in the closure
  of the matrix convex hull of the bent TV screen.
\end{proof}

\subsection{Examples where the Lasserre--Parrilo Lift is  Exact}\label{subsec:74}

Consider first $p=1- x y^2 x$. Then
\[
\cD_p \supseteq \big( \{0\}\times \smat\big) \cup \big(  \smat\times \{0\} \big),
\]
so
$\mco\cD_p$ will equal $\smat^2$.
In particular, the first Lasserre--Parrilo lift  $\hLift_p(\tvs;0)$ is exact.

For an example with a little different flavor, let
$p= (1-2y^2+x^2)\oplus (1-2x^2+y^2).$
Then $\cD_p$ given by
\[
\cD_p=\Big\{ (X,Y)\in\smat^2 :  Y^2- \frac12 X^2  \preceq \frac12,\,  X^2-\frac12 Y^2 \preceq \frac12 \Big\}
\]
is bounded, and
\[
\mco\cD_p=\big\{ (X,Y)\in\smat^2 : \|X\|\leq1, \, \|Y\|\leq1\}
\]
is again the projection of the first Lasserre--Parrilo lift $\Lift_p(\tvs;0)$.


\begin{spacing}{1.1}

\end{spacing}

\end{document}